\documentclass[11pt]{amsart}
\usepackage{amscd}
\usepackage{amsfonts}
 \usepackage{amssymb}
 \usepackage{amsmath}
\usepackage{tikz-cd}
\usepackage{hyperref}

\numberwithin{equation}{section}
\newtheorem{pro}{{Proposition}}[section]
\newtheorem{lemma}[pro]{{Lemma}}
\newtheorem{theorem}[pro]{{Theorem}}
\newtheorem{definition}[pro]{{Definition}}
\newtheorem{remark}[pro]{{Remark}}
\newtheorem{note}[pro]{{Note}}
\newtheorem{corollary}[pro]{{Corollary}}

\newtheorem{exam}[pro]{Example}
\newtheorem{thm:intro}{Theorem}
\newtheorem{cor:intro}{Corollary}

\newtheorem*{thm*}{Theorem }

\newcommand{\ra}{{\rightarrow}}
\newcommand{\lra}{{\longrightarrow}}

\newcommand{\ie}{{that is, }}
\newcommand{\cf}{{cf$.$\,}}

\newcommand{\cN}{{\mathcal{N}}}

\newcommand{\cA}{{\mathcal{A}}}
\newcommand{\cB}{{\mathcal{B}}}

\newcommand{\cG}{{\mathcal{G}}}

\newcommand{\cX}{{\mathcal{X}}}
\newcommand{\cY}{{\mathcal{Y}}}
\newcommand{\cZ}{{\mathcal{Z}}}

\newcommand{\al}{\alpha}
\newcommand{\be}{\beta}

\newcommand{\om}{\omega}

\newcommand{\lam}{\lambda}

\newcommand{\g}{{\mathfrak g}}
\renewcommand{\a}{{\mathfrak a}}
\newcommand{\h}{{\mathfrak h}}
\renewcommand{\k}{{\mathfrak k}}

\newcommand{\s}{{\mathfrak s}}
\newcommand{\n}{{\mathfrak n}}

\newcommand{\CC}{{\mathbb C}}

\newcommand{\HH}{{\mathbb H}}

\newcommand{\RR}{{\mathbb R}}

\newcommand{\ZZ}{{\mathbb Z}}

\newcommand{\SL}{\mathop{\rm SL}\nolimits}

\newcommand{\PSL}{\mathop{\rm PSL}\nolimits}

\newcommand{\Sim}{\mathop{\rm Sim}\nolimits}
\newcommand{\Aut}{\mathop{\rm Aut}\nolimits}

\newcommand{\Psh}{\operatorname*{Psh}\,}
\newcommand{\Pspm}{\operatorname*{Psh^{\pm}}\,}
\newcommand{\Diff}{\operatorname*{Diff}\,}

\newcommand{\Isom}{\operatorname*{Isom}\, }

\newcommand{\fA}{{\mathsf{A}}}
\newcommand{\fB}{{\mathsf{B}}}

\newcommand{\fG}{{\mathsf{G}}}

\newcommand{\fK}{{\mathsf{K}}}

\newcommand{\fR}{{\mathsf{R}}}
\newcommand{\fS}{{\mathsf{S}}}

\renewcommand{\forall}{\text{for all }} 

\begin{document}
\baselineskip 13pt 
\pagestyle{myheadings} 
\thispagestyle{empty}
\setcounter{page}{1}

\title[]{Locally Homogeneous Aspherical Sasaki Manifolds}
\author[]{Oliver Baues}
\address{Department of Mathematics\\ 
University of Fribourg\\
Chemin du Mus\' ee 23\\
CH-1700 Fribourg, Switzerland}
\email{oliver.baues@unifr.ch}

\author[]{Yoshinobu Kamishima}
\address{Department of Mathematics, Josai University\\
Keyaki-dai 1-1, Sakado, Saitama 350-0295, Japan}
\email{kami@tmu.ac.jp}

\date{\today}

\begin{abstract}
Let $G/H$ be a contractible homogeneous Sasaki manifold. 
A compact locally homogeneous aspherical Sasaki manifold
$\Gamma\big\backslash G/H$ is by definition 
a quotient of $G/H$ by a discrete uniform subgroup 
$\Gamma\leq G$.
{We show that a compact locally homogeneous aspherical Sasaki manifold
 is always quasi-regular, that is,
 $\Gamma\big\backslash G/H$ is an $S^{1}$-Seifert bundle over a locally homogeneous 
 aspherical K\"ahler orbifold.}
We discuss the structure of the isometry group $\Isom(G/H)$
for a Sasaki metric of $G/H$ in relation with the pseudo-Hermitian group $\Psh (G/H)$ for the 
Sasaki structure of $G/H$. 
We show that a Sasaki Lie group $G$, when $\Gamma\big\backslash G$ is a compact locally homogeneous aspherical Sasaki manifold, is either the universal covering group of $\SL(2,\RR)$ or a modification  
of a Heisenberg nilpotent Lie group with its natural Sasaki structure. 
In addition, we classify all aspherical  Sasaki homogeneous spaces for semisimple Lie groups. 
\end{abstract}
\maketitle

\section{Introduction} 
Let $M$ be a smooth contact manifold with contact form $\om$.
Suppose that there exists a
complex structure $J$ on the contact bundle ${\rm ker}\, \om$
and that the Levi form $d\om\circ J$ is a positive definite Hermitian form.
Then $\{\om,J\}$ is called a \emph{pseudo-Hermitian structure} on $M$ 
and $\{{\rm ker}\,\om,J\}$ is a $CR$-structure as well.
The pair $\{\om,J\}$ assigns a Riemannian metric $g$ to $M$, where 
\begin{equation}\label{sasakimet}
g=\om\cdot \om +d\om\circ J.
\end{equation}
There are two typical,  closely related,    Lie groups on $(M,\{\om,J\})$. 
The group of pseudo-Hermitian transformations of 
$M$ is denoted by
\begin{equation*}\label{eq:pshH}
\Psh(M)=\{h\in \Diff(M)\mid
h^*\om=\om, \ h_*\circ J={J\circ h_* \ \mbox{on}\ {\rm ker}\, \om} \}. 
\end{equation*}
As usual
$\Isom\,(M)$ denotes the isometry group of $(M,g)$.
Obviously
 $$\Psh(M)\leq\Isom(M) \; .$$
Assume that the Reeb field $\cA$ for $\om$ generates a
one-parameter group $T$ of holomorphic transformations on a
$CR$-manifold $(M,\{{\rm ker}\, \om,J\})$, \ie
$$ T\leq\Psh(M) \, . $$ Then $(M,\{\om,J\})$ is said to 
be a \emph{standard pseudo-Hermitian  manifold}. 
In this case, the vector field $\cA$ is  a Killing field of unit length with respect to $g$, and the Riemannian manifold 
$(M,g)$ is also called a \emph{Sasaki manifold} equipped with \emph{Sasaki metric $g$} and 
\emph{structure field}  $\cA$. 
If $\cA$ is a complete vector field with a global flow $T$ which 
%
acts freely and properly on $M$, 
$(M, \{g, \cA\})$ is said to be a \emph{regular  Sasaki manifold}. 
{Note that the Sasaki metric structure $(M, \{g, \cA\})$ determines the 
standard pseudo-Hermitian structure  $(M,\{\om,J\})$ uniquely.}


\smallskip The pseudo-Hermitian group $\Psh(M)$ and isometry group $\Isom(M)$ of a Sasaki manifold are closely related. 
Since the Reeb vector field $\cA$ is determined by $\om$ alone, we have  
$$h_*\cA=\cA\, , \;  \forall  h\in\Psh(M) \; . $$
  Therefore, the Reeb flow $T$ belongs to the center 
of $\Psh(M)$,
\ie 
\begin{equation}\label{CPsh}
\Psh(M)=C_{\Psh(M)}(T) \; .
\end{equation}
Similarly, if $C_{\Isom (M)}(T)$ denotes the centralizer of $T$ in
$\Isom(M)$, using \eqref{sasakimet},
$$ \displaystyle \Psh(M)=C_{\Isom (M)}(T)$$  
follows easily, as well.

In general, 
the group $\Isom(M)$ acts on the set of Sasaki structures $ \{g, \cA\} $ with fixed metric $g$. Furthermore, if $(M,g)$ is not isometrically covered by a round sphere,  the set of Sasaki structures with metric $g$ either consists of two elements $\{\cA, -\cA\}$, or $M$ is a three-Sasaki manifold, admitting  three linear independent Sasaki structures for $g$.  In the latter case, $M$ is compact with finite fundamental group. For these results,  see  \cite{tanno,kashi,tachiyu}. 
Thus, unless $M$ is compact with finite fundamental group, a complete Sasaki manifold always satisfies  
 $$\Isom(M)= \Pspm(M) =  \{ h \in \Isom(M) \mid h^{*} \cA = \pm \cA \}  \; . $$ 

\smallskip
Call a Sasaki manifold $M$ a \emph{homogeneous Sasaki manifold}
if  $\Psh(M)$ acts transitively on $M$.  Accordingly, a homogeneous space  $G/H$ is called a 
\emph{homogeneous Sasaki manifold} if $G/H$ is a Sasaki manifold and the 
action of $G$ factors over $\Psh(G/H)$. Note that any homogeneous Sasaki manifold is also 
a regular Sasaki manifold. 


\subsection{Locally homogeneous aspherical Sasaki manifolds} 
In the following we shall usually assume that 
$G$ acts effectively on $G/H$ and thereby identify $G$ with a closed subgroup of 
$\Psh(G/H)$ whenever suitable. 

\smallskip 
A  \emph{locally homogeneous} Sasaki manifold  is a quotient space 
 $$ M \, = \, \Gamma\, \big\backslash G/H$$ 
of a homogeneous Sasaki manifold $G/H$ by a discrete 
subgroup $\Gamma$ of $G$.  The manifold $M$ is called \emph{aspherical} if its universal 
cover $X$ is \emph{contractible}. In this paper we take up the structure of compact 
\emph{locally homogeneous aspherical} Sasaki manifolds $M$. 

\smallskip 
Setting the stage for the main structure result 
on compact \emph{locally homogeneous aspherical} Sasaki manifolds,
we note 
the following facts:  

\smallskip 
Let $X =G/H$ be a contractible homogeneous Sasaki manifold.
Then the  Reeb flow $T$ on $X$ is isomorphic to the real line 
 $\RR$ and it is acting freely and properly on $X$. Moreover, the  homogeneous pseudo-Hermitian structure on $X$ induces a unique \emph{homogeneous} K\"ahler structure on the quotient manifold $$ W = X/ T$$  such that the projection map  $X \to W$ is a principal bundle projection which is pseudo-Hermitian (that is, $X \to W$ is horizontally holomorphic and horizontally isometric). With this structure the homogeneous K\"ahler manifold $W$  will be called the \emph{K\"ahler quotient} of $X$. (Compare  
Proposition \ref{pr:noncompact}, Theorem \ref{SasakiKprin})

\smallskip 
Let $W$ be any  K\"ahler manifold. Then we denote $\Isom^{\! \pm}_{\! h}(W)$  the subgroup of $\Isom(W)$ that consists of isometries 
which are either holomorphic or anti-holomorphic. Furthermore,  $\Isom_{\! h}(W)$ denotes the subgroup of holomorphic (or K\"ahler-) isometries of $W$.

\smallskip 
Recall that a Lie group is called \emph{unimodular} if its Haar measure is  biinvariant. Any Lie group $G$ which admits a uniform lattice $\Gamma$ is unimodular. 
The main structure result on locally homogeneous aspherical Sasaki manifolds and their isometry groups 
is stated in the following two results: 

\begin{thm:intro}\label{iso=equal} 
Let $X = G/H$ be a contractible homogeneous Sasaki manifold of a unimodular Lie group $G$. 
 Then the following hold:
\begin{enumerate}
\item The K\"ahler quotient\/ $W$ of\/  $X$ 
 is a product of a unitary space $\CC^{k}$ with a 
 bounded symmetric domain $D$.   
\item The Reeb flow $T$ is a normal subgroup of\/ $\Isom(X)$ and there exists an 
induced quotient homomorphism $$\phi: \Isom(X) \to \Isom^{\! \pm}_{\! h}(W) \; , $$ 
which is onto and maps $\Psh(X)$ onto\/ $\Isom_{h}(W)$ with  kernel $T$. 
\item There exists an anti pseudo-Hermitian involution $\tau$ of $X$ such that 
$$\displaystyle\Isom(G/H)=  \Pspm(G/H) =  \Psh(G/H) \rtimes \langle \tau \rangle  
  \; .$$
\item 
The identity component of the pseudo-Hermitian group of $X$ satisfies 
$$  \Psh(G/H)^{0}  =  \displaystyle\Isom(G/H)^{0} =  \left(\cN\rtimes {\rm U}(k)\right) \cdot  S \; , $$ 
where
$\cN$ is a $2k+1$-dimensional Heisenberg Lie group 
and $S$ is a normal semisimple Lie subgroup which covers 
the identity component  $$S_0 = \Isom(D)^{0}$$ 
of the  isometry group of the  symmetric bounded domain $D$.  
Moreover,  $S$ has infinite cyclic center $\Lambda$,  and $$ S \cap \cN = S \cap T = \Lambda \; .$$

\end{enumerate}
\end{thm:intro}

\smallskip 
Building on Theorem \ref{iso=equal} we can deduce: 

\begin{cor:intro} \label{cor:quasi_reg} 
 Let\/  $M = \Gamma\big\backslash G/H$ be a compact locally homogeneous aspherical
Sasaki manifold. 
Then the  coset space 
$\displaystyle \Gamma\big\backslash G/H$ admits 
an $S^1$-bundle over a locally homogeneous aspherical K\"ahler orbifold 
\begin{equation} \label{eq:reg_fibering}
\begin{CD}
S^1@>>>\Gamma\big\backslash  G/H@>>>\phi(\Gamma)\big\backslash W  \; , 
\end{CD}\end{equation}
in which $S^1$ induces the Reeb field.  In particular,  the Sasaki manifold $M$ is {quasi-regular}.  
\end{cor:intro} 
 
\smallskip 
\begin{remark} \label{rem:seifert}
The bundle in \eqref{eq:reg_fibering}  is called  a \emph{Seifert fibering}. 
Here, some finite covering space $\Gamma_{0 }\big\backslash G/H$, with $\Gamma_{0} \leq \Gamma$ a finite index subgroup, 
is a non-trivial $S^1$-bundle over a  K\"ahler manifold 
$\phi(\Gamma_{0})\big\backslash W$.
Note, in addition, that for any Sasaki manifold $M = \Gamma\big\backslash G/H$ as above, 
$\Psh(\Gamma\big\backslash G/H)^{0}$ contains the flow of the Reeb field.  This flow is a compact one-parameter group $S^{1}$ acting almost freely on $M$ and it is giving rise to the bundle \eqref{eq:reg_fibering}. 
{Moreover, since the Sasaki structure on $M$ arises from a connection form, the K\"ahler 
class of $\phi(\Gamma_{0})\big\backslash W$ represents the characteristic class of the circle bundle.} 
\end{remark}

\smallskip 
We further remark: 
\begin{enumerate}
\item[(5)]  
When the anti-holomorphic isometry $\tau$ of $X$ from Theorem \ref{iso=equal} 
normalizes $\Gamma$, we get 
$
\Isom\,(\Gamma\big\backslash G/H)=\Psh(\Gamma\big\backslash G/H)\rtimes
\ZZ_2$, otherwise we have 
$\displaystyle
\Isom\,(\Gamma\big\backslash G/H)=\Psh(\Gamma\big\backslash G/H)$.
\end{enumerate}

\smallskip 
Let $\cN$ denote the $2n+1$-dimensional Heisenberg group with its natural Sasaki metric. 
Using (5) above we also get: 
\begin{enumerate}
\item[(6)]  There exists a compact locally homogeneous aspherical Riemannian manifold
$$  M  = \pi\backslash\cN \, , $$ whose metric is locally a Sasaki metric (that is, it is induced
 from the left-invariant 
Sasaki metric on $\cN$). But $M$ with metric $g$ is not a Sasaki manifold itself. 
\end{enumerate}

\subsubsection{The case of solvable fundamental group} 
We suppose that the fundamental group of the compact aspherical manifold $M$ is virtually solvable.
In this case, if $M$ supports a locally homogeneous Sasaki structure, then Theorem \ref{iso=equal} 
implies that $M$ is finitely covered by a \emph{Heisenberg manifold} $$  \Delta \,  \backslash \, \cN \; ,$$ 
where $\Delta \leq \cN$ is a uniform discrete subgroup of $\cN$. Moreover,  $M$ 
is a non-trivial circle bundle over a compact  flat  
K\"ahler manifold, which in turn is finitely covered by a complex torus $\CC^k / \Lambda$. 
As a matter of fact,  \emph{any} compact 
aspherical K\"ahler manifold is biholomorphic to a flat K\"ahler manifold (see \cite[Theorem 0.2]{bc} and the references therein). As a consequence, any regular Sasaki manifold $M$ is of the above type as well, and 
it admits a locally homogeneous Sasaki structure: 

\begin{cor:intro} \label{cor:solvable}  Let $M$ be a regular  compact aspherical Sasaki manifold with 
virtually solvable fundamental group. Then the following hold: 
\begin{enumerate} 
\item The manifold $M$ is a circle bundle over a K\"ahler manifold that is 
biholomorphic to a flat K\"ahler manifold.  
\item A finite cover of $M$ is diffeomorphic to a Heisenberg manifold. 
\end{enumerate}  
Moreover, the Sasaki structure on $M$ can be deformed (via regular Sasaki structures) to 
a locally homogeneous Sasaki structure. 
\end{cor:intro} 


\subsection{Contractible Sasaki  Lie groups and compact quotients}
We call a Lie group $G$ a \emph{Sasaki group} if it admits a left-invariant Sasaki structure.
Equivalently,  $G$ acts simply transitively by pseudo-Hermitian transformations  on a Sasaki manifold $X$.

\smallskip 
A prominent example of a Sasaki Lie group is the $2n+1$-dimensional \emph{Heisenberg Lie group} $\cN$. 
The Lie group  $\cN$ arises as  a non-trivial central extension of the form $$ \RR \, \to \, \cN \, \to \,\CC^{n} \; , $$
and a natural Sasaki structure on $\cN$ is obtained by a left-invariant  connection form 
which is   associated to this  central extension. 

More generally, we shall  introduce a family 
of  simply connected $2n+1$-dimensional solvable Sasaki Lie groups
$$ \; \cN(k,l)\; , \; k+l = n,  $$ 
called \emph{Heisenberg modifications}. These groups are deformations of $\cN$ 
in $\cN\rtimes T^k$, where $T^{k} \leq \mathrm{U}(n)$ is a compact torus.  
 (\cf Definition \ref{modicN}).  
 
Another noteworthy contractible Lie group which is Sasaki  is
 $$\widetilde{{\rm SL}(2,\RR)} \;  ,$$ 
 the universal covering group of\/ ${\rm SL}(2,\RR)$. Indeed, take \emph{any} left-invariant 
 metric $g$ on  $\widetilde{{\rm SL}(2,\RR)}$ with the additional property that $g$
 is also right-invariant by the one-parameter 
 subgroup  $\widetilde{{\rm SO}(2,\RR)}$. Then the Riemannian submersion map 
 $$ \widetilde {{\rm SL}(2,\RR)} \; \to  \; \widetilde{{\rm SL}(2,\RR)} \big/ \widetilde{{\rm SO}(2,\RR)} =  \HH^1_\CC$$ 
is defined and it is a principal bundle with group $\widetilde{{\rm SO}(2,\RR)}= \RR$ over a Riemannian homogeneous  space $ \HH^1_\CC$
of constant negative curvature.  The metric
 $g$ defines a unique left-invariant connection form 
 $\omega$,  which satisfies \eqref{sasakimet} and has the property that the Reeb field is left-invariant and tangent to the subgroup $\widetilde{{\rm SO}(2,\RR)}$. The isomorphism classes of Sasaki structures thus obtained are parametrized by the curvature of the base.
\medskip 

As an application of our methods we prove: 
\begin{thm:intro}\label{sasakimodi}
Let $G$ be a unimodular contractible Sasaki Lie group. 
Then as a Sasaki Lie group $G$ is isomorphic to either $\cN(k,l)$ or\/ $\widetilde {{\rm SL}(2,\RR)}$ with one of the left invariant Sasaki structures as introduced above. 
(That is,  $G$ admits a pseudo-Hermitian  isomorphism 
to either $\cN(k,l)$ or $\widetilde{{\rm SL}(2,\RR)}$ with a standard Sasaki structure.)
\end{thm:intro}


\begin{remark} As introduced above the family of all Sasaki Lie groups $\cN(k,l)$ is in one to one correspondence with the set of isomorphism classes of flat K\"ahler Lie groups. Compare Section \ref{sec:Hmod}.
{For a discussion of the structure of flat K\"ahler Lie groups,  
see for example \cite{jh} or  \cite{bdf}.}  
\end{remark}

\begin{remark} When dropping the assumption of contractibility, the compact group 
$\mathrm{SU(2)}$ appears  as another unimodular Sasaki Lie group. This group is fibering over the projective line $\mathrm{P}^{1} \CC$, and the example 
is dual to the  Sasaki Lie group\/ $\widetilde{{\rm SL}(2,\RR)}$.  
The two groups are known to be the only simply connected semisimple Lie groups 
which admit a left-invariant Sasaki structure, \cf\/  \cite[Theorem 5]{bw}. 
\end{remark}


Any Lie group $G$ which admits a discrete uniform subgroup $\Delta$ must be unimodular, and 
if such $G$ admits the structure of a Sasaki Lie group then 
 the quotient manifold 
$$ \Delta \, \big\backslash \,  G$$
inherits the structure of a compact locally homogeneous Sasaki manifold. 

\smallskip 
Thus, combining Theorem \ref{sasakimodi} with Corollary \ref{cor:quasi_reg} we obtain: 

\begin{cor:intro} \label{cor: sasakimodi}
Every compact locally homogeneous 
aspherical Sasaki manifold which is of the form $$\Delta \, \backslash G$$ 
is either a Seifert manifold,  which is an $S^1$-bundle over a  
hyperbolic two-orbifold, or it is a Seifert manifold  which is an $S^1$-bundle over a 
flat K\"ahler manifold (which is a complex  torus bundle over a complex torus). 
\end{cor:intro}

\subsection{Sasaki homogeneous spaces of semisimple Lie groups}
Here we consider the question which semisimple Lie groups act transitively by pseudo-Hermitian transformations on a contractible (or, more generally, aspherical) Sasaki manifold. 
The classification of such groups and of the corresponding homogeneous spaces is contained in Theorem \ref{thm:semisimple_hom0} following below.  

\smallskip Let $D$ be a bounded symmetric domain,  equipped with its natural Bergman Riemannian metric.  Then its isometry group  $$S_0 = \Isom(D)^0$$ is a semisimple Lie group which is called a group of Hermitian type, and 
$$    D =  S_{0}/K_{0}$$ 
is a Riemannian symmetric space with respect to this metric, and a homogeneous K\"ahler manifold, as well.
Moreover, $K_{0}$ is a maximal compact subgroup of $S_{0}$. 

\begin{thm:intro}  \label{thm:semisimple_hom0}
For any symmetric bounded domain $D = S_{0}/K_{0}$, there exists a unique semisimple 
Lie group $S$ with infinite cyclic center, which is covering $S_{0}$, and gives rise to a contractible 
Sasaki homogeneous space 
$$ X_{S} = S/K \, $$ 
with K\"ahler quotient $D$. 
Moreover, any contractible homogeneous Sasaki manifold of a semisimple Lie group is of this type. 
\end{thm:intro} 

{\bf Addendum:} In the theorem, $K$ is a maximal compact subgroup of $S$, and $S \leq \Psh(X_S)$ is acting faithfully on $X_S$. The K\"ahler quotient $D$ is a homogeneous K\"ahler manifold whose complex structure is biholomorphic to a bounded symmetric domain. It carries an invariant symmetric K\"ahler Riemannian 
metric, which is unique up to scaling on irreducible factors of the homogeneous space $D$. 

\smallskip 
As a consequence of  Theorem \ref{thm:semisimple_hom0} any Lie group of Hermitian type acts as a transitive group of isometries on an aspherical Sasaki space:  

\begin{cor:intro} \label{cor:semisimple_hom0} For any semisimple Lie group $S_0$ of Hermitian type, there exists a unique 
Sasaki homogeneous space  
$$ Y = S_{0}/K_{1}  \, ,$$ 
where $Y$ is a circle bundle over
the symmetric bounded domain $D = S_0/K_0$. 
\end{cor:intro}

{Note that, in Theorem \ref{thm:semisimple_hom0} and Corollary \ref{cor:semisimple_hom0} the Sasaki structure on $X$, respectively $Y$,  is unique up to the choice of 
an $S_{0}$-invariant (and also symmetric) K\"ahler metric on $D$.}   

\medskip 
\medskip 
\paragraph{\em The paper is organized as follows}

Starting in   Section \ref{sec:prelim},  we collect and explain 
some useful basic facts on regular Sasaki manifolds, including 
the Boothby-Wang fibration and the join construction. 

In Section \ref{sec:isocontra} we discuss the lifting of K\"ahler isometries 
and the role of gauge transformations in the  Boothby-Wang fibration of a 
contractible Sasaki manifold.

We use these  facts to show that  every contractible homogeneous K\"ahler manifold determines a \emph{unique}  contractible homogeneous Sasaki manifold. Also the associated presentations 
 of a homogeneous Sasaki manifold by transitive groups of pseudo-Hermitian transformation are discussed in Section~\ref{sec:hsasaki}. 

Section \ref{sect:homogeneous_k} is devoted to the study of homogeneous contractible K\"ahler manifolds of unimodular Lie groups. Their classification is derived from the Dorfmeister-Nakajima fundamental holomorphic fiber bundle of a homogeneous K\"ahler manifold. 

The structure of locally homogeneous aspherical Sasaki manifolds is picked up in Section \ref{sec:lchom_sas}.  
We establish in Corollary \ref{cor:quasi_reg} that a compact locally homogeneous aspherical Sasaki manifold is always quasi-regular over a compact  orbifold which is modeled on a homogeneous contractible K\"ahler manifold. 
The relevant global results are summarized in Theorem \ref{iso=equal} and its proof. {We also give the proof of  Corollary \ref{cor:solvable} in Section \ref{sect:solvable_case}, see in particular Proposition \ref{pro:Sasaki_def}.}

In Section \ref{sec:classifications} we turn our interest to the classification problem for global model spaces 
of  locally homogeneous Sasaki manifolds:  in particular,  we classify contractible Sasaki Lie groups and contractible Sasaki homogeneous spaces of semisimple Lie groups. In the course, we prove Theorem \ref{sasakimodi} and 
Theorem \ref{thm:semisimple_hom0}. 

In Section \ref{sec:2}
we construct  further explicit examples of locally homogeneous aspherical Sasaki manifolds.

\smallskip 
Refer to \cite{sa}, \cite{bl}, 
\cite{bg} for background on Sasaki metric structures in general.

\section{Preliminaries} \label{sec:prelim}
Let $X = (X, \{\om,J\})$ be a Sasaki manifold with Reeb flow $T$. 

\subsection{Regular Sasaki manifolds}
The Sasaki manifold $X$ is called regular if the Reeb flow  $T$ is complete and $T$ acts freely and properly on $X$. In this situation, either $T = \RR$, or $T =S^{1}$ is a circle group. Moreover, $$W = X/T$$  is a smooth manifold and $X$ is a principal bundle over $W$ with group $T$. 

\begin{exam} Let $X =G/H$ be a homogeneous Sasaki manifold. Then $X$ is regular. (See \cite{bw} and 
Section \ref{sec:hsasaki} below.)
\end{exam}

For the following, see  \cite{bw}: 

\begin{pro}[Boothby-Wang fibration] \label{pr:bw_fib}
Let $X$ be a regular Sasaki ma\-ni\-fold with Reeb flow $T$. Then there is an associated principal bundle 
$$  T \,  \to \, X \stackrel{q}{\to} \, W $$  
 over a K\"ahler manifold  $(W,\Omega, J)$ such that the induced isomorphism
  $$  q_*: \; {\rm ker}\,\om \, \to \, TW$$  is holomorphic and the K\"ahler form on the base 
is satisfying the equation  
\begin{equation} \label{eq:dexact0}
q^*\Omega=d\om \; . 
\end{equation}
Furthermore, there is a natural induced homomorphism 
\begin{equation} \label{eq:homphi}
\Psh (X) \stackrel{\phi}{\longrightarrow} \Isom_{\! h}(W)  
\end{equation}
with kernel $T$, which is satisfying $q \circ \tilde h = \phi(\tilde h) \circ q$, for all $\tilde h \in \Psh (X)$. 
\end{pro}

With the above conditions satisfied,  
we call $(W,\Omega, J)$ the \emph{K\"ahler quotient} of the regular Sasaki manifold $X$.
Also we let 
$$ \displaystyle{\Isom}_{\! h}(W)={\Isom}(W,\Omega,J)$$  denote 
the group of holomorphic isometries of the K\"ahler quotient $W$.   

\begin{proof}[Proof of Proposition \ref{pr:bw_fib}]
The projection $q$ induces an 
isomorphism $${q}_*:{\ker}\, \om\lra
T W$$ at each point. Since
$\om$ is invariant under $T$, $\om$ induces a well defined 
$2$-form $\Omega$ on $W$ such that 
$$d\om(\cX,\cY)=\Omega({q}_* \cX,{q}_*\cY)\, , $$ 
for all horizontal vector fields $ \cX, \cY\in{\ker}\, \om$ that are horizontal lifts.
As $$\iota_\cA \, d\om=0 \, ,$$ it follows that $d\om={q}^*\Omega$ 
and so $d\Omega=0$.
Since the Reeb flow $T$ is holomorphic on ${\ker}\,\om$, using $J$ on ${\ker}\,\om$, ${q}_*$ 
induces a well defined almost complex structure 
$\hat J$ on $W$ such that $\Omega$ is $\hat J$-invariant.
Since  $J$ is integrable (\ie $[T^{1,0},T^{1,0}]\subseteq T^{1,0}$
for the eigenvalue decomposition ${\rm ker}\,\om\otimes\CC= T^{1,0}\oplus T^{0,1}$),
$\hat J$ becomes a complex structure on $W$.
Hence $\Omega$ is a K\"ahler form on the complex manifold $(W ,\hat J)$. 
To simplify notation, from now on,  the same symbol $J$ is used for the complex structure on $W$, 
for which we require that $q$ is a \emph{holomorphic map} on ${\rm ker}\,\om$, \ie 
the induced isomorphism $ \displaystyle q_*: \; {\rm ker}\,\om \, \to \, TW$ 
satisfies $  q_* \circ J  = J  \circ q_* $.

Since it is  commuting with the principal bundle action of $T$, which is arising from  the Reeb flow,  
each holomorphic  isometry $$\tilde h\in \; \Psh(X) = C_{\Psh(X)}(T)$$ 
induces
a diffeomorphism $h:W\to W$, such that the  diagram 
\begin{equation}\label{eq:prinmap}
\begin{CD}
@.     X @>\tilde h>> X \\
@.       @V{q}VV     @V{q}VV \\
@.      W      @> h>>  W \\
\end{CD}
\end{equation} 
is commutative.  (We briefly verify  that $h^*\Omega=\Omega$ and $h_*\circ J=J\circ h_*$ on $W$: 
Indeed, as $\tilde h^*d\,\om =d\,\om$,  
 it follows by \eqref{eq:dexact0} that $\displaystyle {q}^*(h^*\Omega)=
\tilde h^*{q}^*\Omega={q}^*\Omega$.  
This shows $h^*\Omega=\Omega$.  
Since ${q}_*\tilde J=J \, {q}_*$ on ${\rm ker}\, \om$,
 using \eqref{eq:prinmap} it follows $h_* J \, {q}_*(\cY)=  h_{*} {q}_* \tilde J (\cY) =  {q}_{*} \tilde h_{*}  \tilde J (\cY) =  {q}_{*}  \tilde J   \tilde h_{*}  (\cY)  =  J \, h_*{q}_*(\cY)$, 
$\forall \text{vector fields  }  \cY\in{\rm ker}\, \om$, which are  horizontal lifts for a vector field on $W$.  
So $h_*\circ J=J\circ h_*$ on $W$.) 
Thus $h$ is a holomorphic isometry of $W$. 

Further any lift $\tilde h \in \Psh (X) $ of $h$ is unique up to composition with an element of the Reeb flow: Indeed, suppose that $h = \mathrm{id}_W$.
Since $T$ acts transitively on the fibers, after composition with an element of $T$, we may assume that there exists a fixed point $x \in X$ for $\tilde h$. Moreover, since $\tilde h_* \cA = \cA$, the differential of $\tilde h$ at $x$ is the identity of $T_x X$. 
Now every isometry $\tilde h$ of the Riemannian manifold $X$
is determined by its one-jet at one point $x$. 
Hence, ${\rm ker}\, \phi= T$.
\end{proof}

\subsection{Holomorphic and anti-holomorphic isometries} 
For any Sasaki manifold $X$ with Reeb field $\cA$, we briefly recall the interaction of 
$$\Pspm(X)=  \{ h \in \Isom(X) \mid h^{*} \cA = \pm \cA \} $$ 
with the pseudo-Hermitian structure of $X$. For any pseudo-Hermitian  
structure $\{\om, J\}$, the structure  $\{-\om, -J\}$ is called 
the \emph{conjugate structure}. 
%
Then the group of isometries $\Pspm(X)$ 
permutes the pseudo-Hermitian structure of $X$ and its conjugate: 

\begin{lemma}[Sasaki isometries]  \label{lem:hpm}
Let $X$ be any Sasaki manifold and let $h  \in {\Isom}\,(X)$ satisfy $ h_*\cA=\pm \cA$, where $\cA$ is the Reeb field of $X$. Then  $h^*\om=\pm\om$ and  $h_*J =\pm Jh_* \ \mbox{on} \ {\rm ker}\, \om$. 
\end{lemma}
\begin{proof} 
For any $\cX\in{\ker}\, \om$, the equation $g(h_*\cA,h_*\cX)=g(\cA,\cX)$ shows
\begin{equation*}\begin{split}
0&=g(\cX,\cA)=\om(\cX)\om(\cA)+d\om(J\cX,\cA)=\om(\cX)\\
 &=g(h_*\cX,h_*\cA)=\pm g(h_*\cX,\cA)=\pm\om(h_*\cX) \; .
\end{split}\end{equation*} In particular,  $h_*$ maps ${\ker}\, \om$ onto itself.
As $$ h^*\om(\cA+\cX)=\om(\pm \cA)=\pm \om (\cA+\cX) , $$ we deduce that
 $h^*\om=\pm\om.$ 
Next for any $\cX,\cY\in{\ker}\, \om$,  
\begin{equation*}\begin{split}
d\om(J\cX,\cY)&=g(\cX,\cY)=g(h_*\cX,h_*\cY)=d\om(Jh_*\cX,h_*\cY)\\
 & =d\om(h_* (h^{-1}_*Jh_*)\cX,h_*\cY)=h^*d\om((h^{-1}_*Jh_*)\cX,\cY)\\
 &=\pm d\om((h^{-1}_*Jh_*)\cX,\cY).
\end{split}\end{equation*}By the non-degeneracy of the Levi form $d\om\circ J$ it follows that 
\begin{equation}\label{pmhol}
h_*J =\pm Jh_* \ \mbox{on} \ {\rm ker}\, \om.   \qedhere 
\end{equation}
\end{proof}

\subsection{Join of regular Sasaki manifolds} \label{sect:join} 
We describe in detail a natural procedure which explicitly constructs a new Sasaki manifold from a pair of given regular Sasaki manifolds. 
 {This correponds to a variant of the join construction as is discussed in \cite{bgo} for the compact case. In our context we apply the join in the construction of homogeneous Sasaki manifolds.}

\subsubsection{Sasaki immersions} 
Let $X,Y$ be regular Sasaki manifolds with pseu\-do-Hermitian structures $\{\om, J\}$, $\{ \eta, I\}$,  respectively. 
Also, let  $\cA$, $\cB$ denote the  respective Reeb vector fields on $X$, $Y$.  An immersion of manifolds  $$ \iota: Y \to X$$   such that 
\begin{itemize} \item[i)]  the Reeb vectorfield $\cA$ is tangent to the image $\iota(Y) \subseteq X$ and
\item[ ii)] the tangent bundle of  $\iota(Y)$ satisfies $J  \, T \iota(Y) \subseteq T \iota(Y)$ 
\end{itemize}  
is called a \emph{Sasaki  immersion} if 
\begin{itemize}
\item[iii)] $\{ \eta, I\} = \iota^{*} \{\om, J\}$ 
\end{itemize}  
is satisfied.  That is, for a Sasaki immersion, $\{ \eta, I\}$ is obtained by pullback of  $\{\om, J\}$.
Let $q: X \to W$ and $p: Y \to V$ denote the respective K\"ahler quotients.  Then the Sasaki immersion $\iota$ induces a unique \emph{K\"ahlerian  immersion} $$ j: V \to W \; . $$  such that $j \circ p = q \circ \iota$. Note also that $j$ determines the Sasaki immersion $\iota$ \emph{uniquely} up to composition with an element of the Reeb flow $T$. 

\subsubsection{The join construction and Sasaki immersions} 
Let $$(X_i, \{ \om_i, J_i \}) \,  , \;  i =1,2 \; , $$
be regular Sasaki manifolds with Reeb flows $$ T_i = \{ \phi_{i,t} \}_{t \in \RR} \;  . $$
Furthermore, let   $(W_i,\Omega_i)$ denote the K\"ahler quotients of $X_i$,  and  $$ q_i: X_i \to W_i$$  the corresponding Boothby-Wang fibrations. 
Now consider  $$ \bar T = T_1 \times T_2  =  \{ \left(\phi_{1,s}, \phi_{2,t} \right)  \}_{s,t \in \RR} $$  and 
and define  $\Delta = \{ (\phi_{1,t}, \phi_{2,-t}) \}_{t  \in \RR}$ as the diagonal in $\bar T$. Then put
$$     T = \bar T \big/ \Delta \; . $$  
%
%

\begin{pro}[Join of Sasaki manifolds $X_1$ and $X_2$] \label{pro:join}
There exists a unique regular Sasaki manifold $$ X = X_1 * X_2$$ 
with Reeb flow $T$ and K\"ahler quotient  
$$ q:   \; X_1*X_2 \; \, \to  \; \, W = (W_1 \times W_2, \Omega_1 \times \Omega_2) \; , $$
which admits Sasaki immersions  $\iota_{X_{i}}:  X_{i}  \, \to \,  X_1*X_2 $ such that the diagram 
\begin{equation} \label{eq:join_diagram}
\begin{tikzcd}
X_{i}   \arrow{r}{\iota_{X_{i}}}  \arrow{d}{q_{i}}  &       X_1*X_2 \arrow{d}{q} \\
W_{i}  &    \arrow{l}{\mathsf{pr}_{i} }    W _{1} \times W_{2}  
\end{tikzcd}    
\end{equation} 
is commutative $(i =1,2)$.
\end{pro}
\begin{proof} Observe that, via the product action, $\bar T = T_1 \times T_2$
acts properly  and freely on $X_1 \times X_2 $ 
with quotient map 
$$  \bar q = q_1 \times q_2 :  X_1 \times X_2 \;  \to \; W = W_1 \times W_2   \; 
. $$ Define another quotient  map 
 \begin{equation}  \label{eq:joinp} \mathsf{p}: X_{1}  \times X_{2} \;  \to \;  X := (X_1 \times X_2) / \Delta \; ,  
 \end{equation}
and let  $$ q: X \to W $$ be the induced map such that $\bar q = q \circ \mathsf{p}$.   

Let $\mathrm{pr}_i: X_1 \times X_2 \to X_i$, $i =1,2$,  denote the projection maps. 
Define $$\bar \om = \mathrm{pr}_1^* \, \om_1 + \mathrm{pr}_2^*\, \om_2 \; , $$  
and consider the K\"ahler form  $\Omega = \Omega_{1} \times \Omega_{2}$ on $W$. By construction, 
$$ \bar q^* \, (\Omega)   
= d \bar \om \;   . $$ 


Next let $\bar \cA_{i}$ denote the canonical lifts of the Reeb fields $\cA_{i}$ to $X_{1} \times X_{2}$, where $\bar \cA_{i}$ is tangent to the factor $X_{i}$, respectively. The one-parameter groups generated by these  vector fields are contained in the abelian Lie group $\bar T$. In particular, these vector fields  are $\Delta$-invariant.  Let $\mathcal{V}_{\Delta}$ denote the one-dimensional distribution on $X_{1} \times X_{2}$,  which is spanned by the vector field $\bar \cA_{1} - \bar \cA_{2}$.  Then $\mathcal{V}_{\Delta}$ is vertical (tangent to the fibers) with respect to the 
quotient map $\mathsf{p}$ in \eqref{eq:joinp} induced by the action of $\Delta$. 
Therefore,  both vector fields $\bar \cA_{i}$ 
project to the same vector field  $\cA$  on $X$. 

Note that $\bar \om$ is a $\bar T$-invariant one-form which vanishes on 
$\mathcal{V}_{\Delta}$. 
Therefore,  there exists on $X$ a unique induced one form 
$$\om = \om_1 * \om_2  \;  \text{ satisfying  $\mathsf{p}^{*} \om = \bar \om$}. $$ 
In particular,  $\om$ satisfies $q^* \Omega = d\om $, where $\Omega = \Omega_{1} \times \Omega_{2}$. 
It follows that  $\omega$ is a contact form with 
Reeb field $\cA$.  
The Reeb flow of $\om$ is the one-parameter group  $$ T = \bar T/ \Delta \, . $$ 
{Summarizing the construction, we note that $\om$ is a connection form for the $T$-principal bundle $q: X \to W$ and it has curvature form $\Omega$.} 

Let $J_{i}$ denote the complex structures on  $\ker \omega_{i}$ (canonically extended to tensors on $X_{i}$ by declaring $J_i(\cA_i) = 0$). 
 Observe that  the kernel of $\om$ coincides with the projection of   
 $$\ker \mathrm{pr}_1^* \, \om_1 \cap  \ker \mathrm{pr}_2^*\, \om_2 $$ 
 to (the tangent bundle of) $X$.  Therefore,  $\bar J =  J_1 \times J_{2}$   goes down to an {almost  complex structure $J$ on $\ker \om $ such that 
$$ q: \; (X, \{ \ker \om, J\}) \, \to \, (W, J) $$ is a holomorphic CR-map. 
Since    $(W, \{ \Omega,J\})$ is K\"ahler and $\omega$ a connection form with 
curvature $\Omega$, the almost CR-structure $\{\ker \om, J\}$ is integrable, see
\cite[Theorem 2]{hatake}.  Since $$ d\om  \circ J = (q^{*} \Omega) \circ J$$ is positive,  
$\{\om, J\}$ defines a pseudo-Hermitian structure on $X$.  By the construction $T$ acts by holomorphic transformations on $X$. 
This shows that $(X, \{\om, J\})$ is a regular Sasaki manifold with K\"ahler quotient $(W, \Omega)$.} 

Choose a base point $(x_{o}, y_{o}) \in X_{1} \times X_{2}$ and define 
immersions $\iota_{i}: X_{i} \to  X$, $\iota_{1} (x) =  q (x, y_{o})$ and  
 $\iota_{2}(y) = q(x_{o}, y)$.  (Note that all such pairs of maps are equivalent by an element of $T$.) 
 By the above 
 construction,  $\iota_{i}$ are Sasaki immersions, and,  in fact,  
 they determine the Sasaki structure $ \{\om, J\}$ on the manifold 
 $X_{1} * X_{2}$ uniquely, together with the condition that $\cA$ is the Reeb field. 
 \end{proof} 
 
 The join of Sasaki manifolds enjoys the following functorial property: 
 
\begin{pro}   \label{pro:join_funct}
For any  pair of Sasaki immersions $\tau_{i}:  Y_{i} \to X_{i}$ 
with induced K\"ahler immersions $j_{i}: V_{i} \to W_{i}$, $i = 1,2$, there exists a  unique 
Sasaki immersion $$ \tau = \tau_{1} * \tau_{2}:  \; Y_{1} * Y_{2}  \; \to \;  X_{1} * X_{2}$$  such that the associated diagram 
\begin{equation} \label{eq:join_functorial}
\begin{tikzcd}
Y_{1} * Y_{2}   \arrow{r}{\tau}  \arrow{d}{p}  &       X_1*X_2 \arrow{d}{q} \\
V_{1}  \times V_{2}   \arrow{r}{ j_{1} \times j_{2} }    &  \; \; \;  W \; \;  
\end{tikzcd}    
\end{equation} 
is commutative and  $\iota_{X_{i}}  \circ \tau_{i}$ and $ \tau  \circ \iota_{Y_{i}}$ coincide up to an element of 
$T$.  
\end{pro}   
\begin{proof}
 Since $\tau_{i}$ are Sasaki immersions, the product map  $$ \bar \tau = \tau_{1} \times \tau_{2}: \;  Y_{1} \times Y_{2} \to X_{1} \times X_{2}$$  induces a map 
$$ \tau_{1} * \tau_{2} :=  Y_{1}*Y_{2} \to  X_{1} * X_{2}  $$ 
with the required properties. 
\end{proof}


This gives:

\begin{corollary} The join of $X_{1}$ and $X_{2}$ defines a natural homomorphism 
$$   \Psh(X_{1}) \times \Psh(X_{2}) \,  \to  \, \Psh(X_{1}  * X_{2})  \,\,  , \;  (\phi_{1}, \phi_{2}) \mapsto \phi_{1} * \phi_{2  }$$
with kernel the diagonal group $\Delta =  \{ (\phi_{1,t}, \phi_{2,-t} ) \}_{t \in \RR}$. 
\end{corollary} 
\begin{proof} Indeed, by the construction in Proposition \ref{pro:join_funct}, $\phi_{1} * \phi_{2} \in \Psh(X)$ and the above map is a homomorphism with kernel $\Delta$.
\end{proof}

\noindent We call  the group 
$$\Psh(X_{1}) * \Psh(X_{2}) \,  = \;   \left(  \Psh(X_{1}) \times \Psh(X_{2}) \right)  \big/ \, \Delta $$
the \emph{join}  of the groups $\Psh(X_{i})$. By the above, the 
join of $\Psh(X_i)$ identifies with a subgroup of  $\Psh(X_{1}  * X_{2})$.
\begin{corollary} \label{cor:join}
Let $X_{1}$ and $X_{2}$ be  homogeneous Sasaki manifolds. 
Then the join of groups $ \Psh(X_{1}) * \Psh(X_{2})$ is acting transitively 
 by pseudo-Hermitian transformations on the Sasaki manifold $X_{1}  * X_{2}$. 
 In particular,  $X_{1}  * X_{2}$ is a  homogeneous Sasaki manifold.
\end{corollary} 
\begin{proof} The K\"ahler quotient $W_{1} \times W_{2}$ of $X_{1}*X_{2}$  is a homogeneous K\"ahler manifold for the group $G = G_{1} \times G_{2}$, where $G_{i}$ denotes the Boothby-Wang image of  $ \Psh(X_{i})$  in $\Isom_{\! h}(W_{i})$. Since $G$ is also the Boothby-Wang image of 
$ \Psh(X_{1}) * \Psh(X_{2})$, and the latter also contains the Reeb-flow $T$, it follows that
$\Psh(X_{1}) * \Psh(X_{2})$ acts transitively on $X_{1}*X_{2}$. 
\end{proof}

\section{Pseudo-Hermitian group $\Psh(X)$ of a regular Sasaki manifold with vanishing K\"ahler class}\label{sec:isocontra}

Suppose that $(X,\{\om,J\})$ is a regular Sasaki manifold with Reeb flow $T$ isomorphic to the real line $\RR$. 
Then the Boothby-Wang fibration Proposition \ref{pr:bw_fib} gives  a principal bundle
\begin{equation*}\begin{CD}
\RR@>>> X@>q>> W
\end{CD}\end{equation*}
over the  K\"ahler quotient $W = (W,\Omega,J)$. Here the 
group $\RR =  \{\varphi_t\}_{t\in\RR}$ of the principal bundle is generated by the Reeb field and  the K\"ahler form on the base 
is satisfying the equation  
\begin{equation}\label{eq:dexact}\begin{CD}
q^*\Omega=d\om \; . 
\end{CD}
\end{equation}

\smallskip 
Choose a smooth section $s:W \to  X$ of $q$
such that the bundle $X$ is equivalent to the trivial bundle
by a bundle map $$ \displaystyle f:\RR\times W \lra X$$
which is defined by $$ f(t,w)=\varphi_t \, s(w) \; . $$
We thus have the following
commutative diagram:
\begin{equation}\label{eq:commuta}
\begin{CD}
 @. \RR\times W@>f>> \;  \;  \; X  \\
@.     {\rm pr}\searrow @.  
\! \! \swarrow {q} @ .   \\
@.   &  W \  @. \\
\end{CD} \; \;  \; \;  \; , \;   \text{ $q \circ s = \mathrm{id}_{W}$ }  . 
\end{equation}
Declare a one-form $\theta$ on $W$ by putting 
\begin{equation}\begin{CD} \label{theta}
\theta=s^*\om \; .
\end{CD}
\end{equation} Note then that  $d\theta=\Omega$ from 
\eqref{eq:dexact}. 
In particular,  \emph{the K\"ahler form $\Omega$ on  $W$ is exact}. 

\smallskip 
Next  extend $\theta$ to a translation invariant one-form on $\RR\times W$ by declaring 
\begin{equation}\label{eq:canoform} \begin{CD}
\om_0=dt+{\rm pr}^*\theta, \;  \mbox{ so that}\ d\om_0={\rm pr}^*\Omega \text{ holds}.
\end{CD}
\end{equation}
Noting $\displaystyle f(0,w)=s(w)={s}\circ{\rm pr}(0,w)$,
we have 
$${\rm pr}^*\theta_{|_{\{0\}\times W}}=
\left((s\circ{\rm pr})^*\om\right)_{|_{\{0\}\times W}}=f_{|_{\{0\}\times W}}^*\om.$$
Since both forms $f^*\om$ and $\omega_{0}$ are translation invariant, we  conclude that 
\begin{equation}\label{eq:pullform}
\begin{CD}
f^*\om=\om_0  \, . 
\end{CD}
\end{equation}
Then an almost complex structure $\tilde J$ on ${\rm ker}\, \om_0$ is defined by
\begin{equation}\label{alcomp}
{\rm pr}_*\circ \tilde J=J\circ{\rm pr}_*.
\end{equation}
By construction, the isomorphism $\displaystyle f_* : {\rm ker}\, \om_0\to {\rm ker}\, \om$
is holomorphic, \ie
$$ \displaystyle f_*\circ \tilde J=J\circ f_* .$$
In particular, $\tilde J$ is a complex structure on ${\rm ker}\, \om_0$.
Summarizing the above we obtain: 

\begin{pro}\label{pr:oductiden}
Identifying the regular Sasaki manifold $X$ with $\RR\times W$ via $f$,
the pseudo-Hermitian structure $\{\om,J,\cA\}$ corresponds to 
$\{\om_0, \tilde J,{\partial \over {\partial t}}\}$ on  the trivial bundle  $\RR\times W$, where $\om_0$ is defined as in \eqref{eq:canoform}.
\end{pro}

\paragraph{\em Existence of  a compatible regular Sasaki manifold} 
Conversely, any \emph{exact}  K\"ahler form $$ \Omega = d\theta$$  
on a complex manifold $W$ arises as the curvature form of 
a \emph{connection form}  $\omega$ on the  trivial principal 
bundle  $$ X = \RR \times W \; . $$ 
In fact, such $\om$ with Reeb field $\cA = {\partial \over {\partial t}}$ is given by  \eqref{eq:canoform}. 
As a consequence (employing \cite[Theorem 2]{hatake} to show the integrability of the almost 
CR-structure $\{\ker \om, J \}$), 
there exists on $X$ a pseudo-Hermitian structure 
\begin{equation}  \label{eq:odtheta} 
  \{\om,J,\cA\}  \; ,
\end{equation}
which has the K\"ahler manifold $(W, \Omega)$ as its K\"ahler quotient. 
We call such a pseudo-Hermitian structure \emph{compatible with the K\"ahler manifold $(W, \Omega)$}.

\smallskip  
We remark now that,  under a mild assumption 
 on the K\"ahler manifold $W$, any compatible pseudo-Hermi\-tian structure on $X$ is essentially determined uniquely 
by  the K\"ahler structure on $W$.  
 
 \begin{pro}\label{pr:gauge}
Suppose $H^{1}(W,\RR) = \{ 0 \}$. Then any two pseudo-\-Hermi\-tian structures 
$\{\om,J,\cA\}$  and  $\{\om',J',\cA\}$  on $X$,  which are 
compatible with the K\"ahler manifold $(W, \Omega)$, are related by a
gauge transformation for the principal bundle  $q: X \to W$.
\end{pro} 
\begin{proof} By the compatibility assumption, we have  
$\om' - \om = q^{*} \eta$, for some closed one-form $\eta \in \Omega^{1}(W)$.
Since $H^{1}(W,\RR) = \{ 0 \}$, there exists a function $\lambda: W \to \RR$ such  $\eta = d \lambda$.
In the view of  Proposition \ref{pr:oductiden},  we may assume that $X = \RR \times W$ 
and $\omega = dt + q^{*} \theta$, where $d \theta = \Omega$. We define  a gauge transformation $G$ for the bundle $q$, by putting
$$G(t,w) = (t + \lambda(w), w) \; . $$ 
We then calculate $
G^* \om  = G^{*}dt + q^{*} \theta  = dt +  d\,  q^*\lam +  q^{*} \theta  = \omega +  q^{*} \eta = \om' $.
\end{proof}

\begin{remark} \label{kaehler_to_sasaki} 
\em For an analogue existence result for Sasaki manifolds in the 
more elaborate case of circle bundles over 
Hodge manifolds,  see \cite[Theorem 3]{bw}, respectively \cite{kob}. 
\end{remark}

\subsection{Lifting of isometries  from the K\"ahler quotient}
We now prove a structure result for the group of holomorphic isometries 
$\Psh (X)$ of $X$ if the Boothby-Wang fibration 
has contractible fiber $\RR$. That is,  let $X$ be a regular Sasaki manifold 
with Boothby-Wang fibration 
\begin{equation}   \label{eq:kahlerq}
 \displaystyle \RR\to X  
\to W \;  . 
\end{equation} 
As before let 
$$ \displaystyle{\Isom}_{\! h}(W)={\Isom}(W,\Omega,J)$$  
denote the group of holomorphic isometries of the K\"ahler 
quotient $W$ for $X$.   

\begin{pro}\label{pr:noncompact}
Assume  that  the  first cohomology of\/ the K\"ahler quotient $W$, arising in \eqref{eq:kahlerq}, 
satisfies $H^{1}(W) = \{0\}$. 
Then the Boothby-Wang homomorphism \eqref{eq:homphi} defines a natural  exact sequence 
\begin{equation}\begin{CD}\label{eq:const}
1@>>> \RR@>>>\Psh (X)@>\phi>>\Isom_{\! h}(W)@>>>1 \;  .
\end{CD}\end{equation}
In particular, $\Psh (X)$ acts transitively on $X$ if and only if\/  
$\Isom_{\! h}(W)$ acts transitively on $W$.
\end{pro}

\begin{proof}
In the view of Proposition \ref{pr:bw_fib} it is sufficient to show that $\phi$ is surjective. 
Indeed, since $H^{1}(W) = \{ 0 \}$,
Lemma \ref{lem:iso_lift} below shows that for any $h\in\Isom_{\! h}(W)$, 
there exists an isometry $\tilde h\in \Psh (X)$, which is a lift of $h$, \ie $\phi(\tilde h) = h$. 
\end{proof}

\smallskip 
Proposition \ref{pr:noncompact} is implied by the following 
basic lifting result for holomorphic and anti-holomomorphic isometries of the K\"ahler quotient 
$W$:
\begin{lemma} \label{lem:iso_lift}
Assume  that $H^{1}(W) = \{0\}$, and let $h \in\Isom (W)$ satisfy 
 $h^* \Omega = \mu \,  \Omega$, where $\mu \in \{\pm 1\}$.
Then there exists an isometry $\tilde{h} \in  \Pspm(X)$
 such that $\tilde h$ induces 
$h$ on  $W$ and satisfies $\tilde h^*\om  = \mu \,  \om$. 
If $\mu = 1$ then $\tilde h\in \Psh (X)$. 
\end{lemma}
\begin{proof} 
We may assume
$X=\RR\times W$. Define $\tilde h' (t,w) = (t, h(w))$ to be the
canonical lift of $h$. 
Then $\om' =  \mu \cdot  (\tilde h')^{*}   \om$ defines another 
pseudo-Hermitian structure on $X$ which is compatible with 
$(W, \Omega)$. By Proposition \ref{pr:gauge}, there exists 
a gauge transformation $G: X \to X$ with $G^{*} \,  \om' = \om$.
Therefore, $$ \tilde h = G \circ \tilde h'$$ 
satisfies $\tilde h^{*}  \om  = \mu \cdot \om$,  and it is an
 isometric lift of $h$ for the metric 
 $g = \om  \cdot \om + d\omega \circ J$. 
 It also follows $\tilde h^{*}  \cA = \mu \cA$. Thus, $\tilde h \in  \Pspm(X)$.
\end{proof}


\section{Homogeneous Sasaki manifolds}\label{sec:hsasaki}
Suppose  that the Lie group $G$ acts transitively by pseudo-Hermitian isometries on the Sasaki manifold $X$.
Then 
$$ X = G/H$$  is  called a homogeneous Sasaki manifold.  
Since $X$ is also a complete Riemannian manifold with respect to the Sasaki metric $g$, the Reeb field $\cA$ for $X$, which is a Killing field for the metric $g$, is a complete vector field.  
Let $$ T=\{\varphi_t\}_{t\in \RR}$$  denote the 
$1$-parameter group on $G/H$ generated by the Reeb field.

\subsection{Natural fibering over homogeneous K\"ahler manifold}  \label{sec:bwhomogeneous}
Since $T$ commutes with $G$, there exists a one-parameter
subgroup 
\begin{equation} \label{eq:sfA} 
{\sf A}=\{a_t\}_{t\in \RR} \; \leq N_G(H)
\end{equation} such that 
\begin{equation} \label{right}
\varphi_t(xH)=xa_t^{-1}H \; ,
\end{equation} where $N_G(H)$ denotes the normalizer of $H$ in $G$. 

\begin{pro}\label{closed}
$T$ is a closed subgroup in $\Psh(G/H)$.
In particular,
$T$ is isomorphic to $S^1$ or $\RR$, and it is
acting properly on $G/H$. 
\end{pro}

\begin{proof}The Reeb field $\cA$ is uniquely determined by the equations:
\[\om(\cA)=1,\ \ \iota_{\cA}d\om=0.\]
Let $\displaystyle D=\overline{\{\varphi_t\}}_{t\in \RR} \leq \Psh(G/H)$ be the closure. 
As $L_g\varphi_t=\varphi_tL_g$ $(\forall\,g\in G)$ from \eqref{right},
every element of $D$ commutes with $G$.
Thus every vector field $\cB$ induced from one-parameter groups
in $D$ is left-invariant. 
In particular, $\om(\cB)$ is constant. By the Cartan formula, 
it follows $\displaystyle\iota_{\cB}d\om=0$.
If $\om(\cB)\neq 0$, by uniqueness of the  Reeb field,  $\cB=\cA$ up to a constant multiple on $G/H$. 
When $\om(\cB)=0$, the non-degeneracy of the Levi form $d\om\circ J$ on
${\rm ker\, \om}$ implies $\cB=0$ on $G/H$. 
This shows $\displaystyle D=\{\varphi_t\}_{t\in \RR}$.
\end{proof}

\begin{lemma}\label{lem:free}
$T$ acts freely on $G/H$.
\end{lemma}

\begin{proof}
If $\varphi_{t_0}(x_0H)=x_0a_{t_0}^{-1}H=x_0H$,   
for some $x_0\in G$, then $a_{t_0}\in H$ and so
$\varphi_{t_0}(xH)=xH$\ $(\forall\,x\in G)$.
Since $T$ acts \emph{effectively},
$\varphi_{t_0}=1$.  
\end{proof}

\smallskip
In particular, any homogeneous Sasaki manifold $X= G/H$  is 
a regular Sasaki manifold (\cf \cite{bw}).  
Moreover, by Proposition \ref{pr:bw_fib} the K\"ahler quotient  
$$W = (G/H) \big /\, T$$ is a homogeneous 
K\"ahler manifold for $G$. That is,  $G$ is acting transitively by holomorphic isometries on $W$. We thus have: 

\begin{theorem}[Boothby-Wang fibration \mbox{\cite{bw}}] \label{SasakiKprin}
Every homogeneous Sasaki  manifold  $X= G/H$
arises as 
a principal $T$-bundle over a
homogeneous K\"ahler manifold $W$ which takes the form:  
\begin{equation}\label{eq:sasakifibK}
\begin{CD}
T@>>> G/H@>{\sf q}>> W = G/ H {\sf A}  \; .
\end{CD}\end{equation}
\end{theorem}


%

\smallskip 
\begin{remark} If $G/H$ is contractible, so is $G/ H {\sf A} $, and in this case $T \cong \RR$. 
\end{remark} 

\smallskip
The following \emph{existence and uniqueness result}  for contractible homogeneous   Sasaki manifolds  is  now a direct consequence of Section \ref{sec:isocontra}: 

\begin{corollary}[Contractible homogeneous Sasaki manifolds] \label{cor:hom_sasaki}
Let $(W,\Omega, J)$ be a homogeneous K\"ahler manifold which is contractible. Then there exists a contractible homogeneous Sasaki manifold $(X, \{\om, J\})$ which has K\"ahler quotient $(W,\Omega, J)$. Moreover, with these properties, the Boothby-Wang fibration \eqref{eq:sasakifibK} for $X$ has fiber $\RR$, and $X$ is uniquely defined up to a pseudo-Hermitian  isometry.  
\end{corollary}
\begin{proof} Indeed, we may choose on the trivial principal bundle $X = \RR \times W$, the pseudo-Hermitian structure \eqref{eq:odtheta}, which has Reeb field $\cA= {\partial \over \partial t}$ and K\"ahler quotient $(W,\Omega, J)$.  By Proposition \ref{pr:noncompact}, $(X, \{ \om, J, \cA \})$ is a homogeneous Sasaki manifold. Let $(X',  \{ \om', J', \cA' \})$ be another contractible Sasaki manifold which has  $(W,\Omega, J)$ as a K\"ahler quotient. Then the Boothby-Wang fibration for $X'$ has fiber $\RR$, and,  by Proposition \ref{pr:oductiden},  
there exists a  pseudo-Hermitian isometry from $X'$ to 
$(X,  \{ \om', J', \cA\})$. By Proposition \ref{pr:gauge}, the latter admits a  pseudo-Hermitian 
isometry  to $(X, \{ \om, J, \cA \})$ which is given by a gauge transformation of the bundle $X$. This implies the claimed uniqueness. 
\end{proof}

%

\subsection{Pseudo-Hermitian presentations of  $W$} \label{sec:presentations}
Let $X$ be a homogeneous Sasaki manifold with group $G$ and $W$ its K\"ahler quotient. We describe now the types of homogeneous presentations  $$W = G/H {\sf A}  $$ which can  arise in the associated Boothby-Wang fibration  \eqref{eq:sasakifibK}. For this we assume that 
$$   G \, \leq \, \Psh(X) $$ 
is a closed subgroup. In particular, 
$G$ is acting faithfully on $X$. With this assumption the  stabilizer $H$ is always compact, since $G$ is a closed group of isometries for $X$. 

\smallskip
\begin{lemma} \label{lem:Delta}
Let $ \Delta$ denote the kernel of the induced $G$-action on the K\"ahler quotient $ W$ of $X$. Then  the following hold:
\begin{enumerate}
\item $ H \, {\sf A}  =  H \rtimes {\sf A} $ decomposes as a semi-direct product.
\item  $\Delta  \leq  H {\sf A} $, and,  $\bar L =  H  {\sf A}  \big/ \Delta$ is  compact. 
\item  $\Delta = T \cap G$, in particular, $\Delta$ is central in $G$. 
\item  If\/ $\fA$ is non-compact then the projection homomorphism $\pi_{\sf A}: H {\sf A}  \to  {\sf A}$ maps $\Delta$ injectively to a closed  subgroup of $ {\sf A}$. 
\item If $\fA$ is normal in $G$ then $\fA$ is central in $G$.
\end{enumerate}
\end{lemma} 
\begin{proof}
Since $T$ acts freely on $G/H$, we infer from \eqref{right} that ${\sf A} \cap H = \{ 1\}$. 
This implies that  $$ H \, {\sf A}  =  H \rtimes {\sf A} $$ is  a semi-direct product, proving (1).  Let \begin{equation*} 
	\pi_{\sf A}: H {\sf A}  \to  {\sf A} 
\end{equation*}  denote 
the projection homomorphism. Since $H$ is compact, the homomorphism $\pi_{\sf A}$ is proper.  
Therefore,  the image $\bar G$ of $G$ in $\Isom(G\big/H {\sf A} )$ is closed  and acts properly on 
$ W = G/ H {\sf A}  = \bar G/ \bar L $.  We deduce that  $\bar L  =  H  {\sf A}  \big/ \Delta$ is  a compact subgroup of $\bar G = G/\Delta$. Thus, (2) holds.

Since the homomorphism $\phi$ in \eqref{eq:homphi}  which maps $G$ to $\bar G$  has kernel $T$, $$ \Delta = G \cap T \, , $$  where the intersection is taken in $\Psh(X)$.  Recall that $T$ is central in $\Psh(X)$. Therefore, $\Delta$ is central in $G$. Hence, (3).

Next, consider 
$ C = \ker \pi_{\sf A} \cap \Delta = H \cap \Delta $. Assuming that $\fA$ is a vector group, $C$ is  the unique maximal compact subgroup of $\Delta$. Since $\Delta$ is normal in $G$, so is $C$. Since $C$ is also a subgroup of $H$ and $G/H$ is  effective, we deduce that $C = \{1 \}$. This shows that  $\Delta$ is isomorphic to the closed subgroup $\pi_{\sf A}(\Delta) \leq  {\sf A}$, proving (4).

Finally,
%
%
 assume that   $\fA$ is normal in $G$. Then the left-multiplication orbits of $\fA$ on $G/H$ coincide with the 
orbits of $T$. That is, for all $g \in G$:  $$ T \cdot g H =  g \cdot \fA H = \fA \cdot g H  \; . \text{} $$ 
In particular, the left-action of $\fA$ on $G/H$ (which is by pseudo-Hermitian isometries)  induces the trivial action 
on the K\"ahler-quotient $W$ by the fibration sequence \eqref{eq:sasakifibK}. That is, 
$\fA \leq \Delta$  and by (3),  $\fA \leq T$. This implies that $\fA = T$ is central in $G$. 
%
\end{proof}

Two principal cases are arising,  according to whether $\Delta$ is a continuous group  or $\Delta$ is a discrete subgroup of $G$. Recall first that either $\fA = S^1$ or $\fA = \RR$. Then we have: 

\smallskip 
\paragraph{\bf Case I ($\mathbf{\Delta = \fA}$, $T$ is contained in $G$)} We suppose here that $\fA$ can be chosen to be a normal subgroup in $G$.  
By (5) of Lemma \ref{lem:Delta}, it follows that the isometries induced by the left-action of $\fA$ are contained in  the kernel of the homomorphism 
$\phi: \Psh(X) \to \Isom_{\! h}(W)$, which is just $T$. 
Since $\fA$ is a non-trivial connected (one-dimensional) group, this implies $$T = \fA = \Delta$$ as subgroups of $\Psh(X)$. 
Then the fibration   \eqref{eq:sasakifibK} turns into a  principal bundle of homogeneous spaces of the form
\begin{equation}  \tag{I} \label{eq:sasakifibKI}  
\begin{CD}
{\sf A}@>>> G/H@>{\sf q}>> W= \left(G/{\sf A}\right) /H = \bar G / \bar H  \; .
\end{CD}\end{equation}
where $\bar H = H$ and the group $\bar G$ is described  by  an exact sequence of groups
\begin{equation}  \tag{I'}  \label{eq:sasakifibGI} %
\begin{CD} 
1@>>> \fA =\RR  @>>> G@>{\phi}>> \bar G @>>>1 \;  .
\end{CD}
\end{equation}

%
%


\smallskip 
\paragraph{\bf Case II ($\mathbf{ {\sf A} = \RR, \Delta = \ZZ}$)}  We are  assuming that  ${\sf A} \cong \RR$ (for example, if  $G/H$ is contractible). By Lemma \ref{lem:Delta} (4), the central subgroup $\Delta$ of $G$ is either infinite cyclic (and discrete) or $\Delta$ is a closed one-parameter subgroup in $ H  {\sf A} $ which is 
is projecting surjectively onto ${\sf A}$.  Since $\Delta$ is contained in $T$, and $T$ is one dimensional, we 
deduce $\Delta = T$, in the latter case. This situation was already described in Case I above. 

\smallskip 
So for case II,  $\Delta = T \cap G$ is infinite cyclic and central in $G$. 
Moreover, $\Delta \leq  H \fA$ and by Lemma \ref{lem:Delta} (4) 
the map $\pi_{\fA}$  is projecting $\Delta$ injectively onto a 
discrete lattice $\mathcal{Z}$ in ${\sf A}$. 
Denote with $\bar \fA$ the image of $\fA$ in $\bar G = G/\Delta$. 
Then the Boothby-Wang fibration \eqref{eq:sasakifibK} can be written in  the form  
\begin{equation}  \tag{II} \label{eq:sasakifibKII}
\begin{CD}
{\sf A}@>>> G/H@>{\sf q}>> W=  \bar G\big/ \bar H \bar {\sf A}   \; .
\end{CD}
\end{equation}
where the group $\bar G$ is described  by the exact sequence 
\begin{equation} \tag{II'}   \label{eq:sasakifibGII}
\begin{CD} 
1@>>> \Delta = \ZZ  @>>> G@>{\phi}>> \bar G @>>>1 \;  . 
\end{CD}
\end{equation}

\smallskip  
Recall also that $\bar L =  \bar H \bar {\sf A} $ is a compact subgroup of $\bar G$,
and $\bar H$ is a compact normal subgroup in $ \bar H \bar {\sf A}$.
Therefore, the simply connected one-parameter group  $\fA$ may be chosen in such a way that its quotient $\bar \fA$ is a compact circle group, and the intersection $\bar H \cap \bar A$ is finite. 
%
\section{Homogeneous K\"ahler manifolds of unimodular groups} \label{sect:homogeneous_k} 

Let $W$ be a homogeneous K\"ahler manifold. 
The fundamental conjecture  for homogeneous K\"ahler manifolds  (as proved by Dorfmeister and Nakajima \cite{DN}) 
asserts  
that $W$ is a holomorphic fiber bundle 
over a homogeneous bounded domain $D$ 
with fiber the product of a flat space $\CC^k$ with a compact simply connected homogeneous 
K\"ahler manifold. 

\smallskip 
Recall that a Lie group $G$ is called \emph{unimodular} if its Haar measure is biinvariant. Let $\g$ denote the Lie algebra of $G$. If $G$ is connected, then $G$ is unimodular if and only if the trace function over the adjoint representation of $\g$ is zero.

\begin{pro}  \label{pro:hom_Kaehler}
Let $W$ be a contractible homogeneous K\"ahler manifold that admits a connected unimodular subgroup 
$$ G \leq \Isom_{\! h}(W)$$  which acts transitively on $W$. 
Then 
there exists a symmetric  bounded domain $D$ such 
$$W = \CC^k \times D$$  is a K\"ahler direct product. 
\end{pro}
\begin{proof} For the proof of the proposition we require some constructions  which are developed in the proof of the fundamental conjecture as it is given in  \cite{DN}. 
The first main step in the proof is to modify $G$ in order  to obtain a 
suitable connected transitive Lie group $\hat G$ with particular nice properties \cite[Theorem 2.1]{DN}. 
By a modification procedure on the level of Lie algebras (as is described in \cite[\S 2.4]{DN}),  
we obtain from the K\"ahler Lie algebra $\g$ of $G$ a quasi-normal 
K\"ahler Lie algebra $\hat \g$. Moreover, it is shown that there exists a connected subgroup  
$\hat G \leq \Isom_{\! h}(W)$,  which has Lie algebra $\hat \g$ and acts transitively on $W$. 
As can be verified directly from  \cite[\S 2.4]{DN},  
the modified Lie algebra $\hat \g$ preserves unimodularity of $\g$ and also satisfies $\dim \hat \g \leq \dim \g$. 

Therefore, from the beginning,  we may assume that the connected unimodular transitive Lie group
 $G$ of holomorphic isometries in question has quasi-normal Lie algebra $\g$. 
We can also replace $G$ with its universal covering group, and we 
remark that \emph{$K$ is connected}  ($W= G/K$  is simply connected, since we are assuming here that $W$ is contractible). 
With these additional properties in place, according to \cite[Theorem 2.5]{DN} combined with  \cite[ \S 7]{DN}, 
the following hold: 
\begin{enumerate} 
\item There exists a closed connected 
normal abelian subgroup $A$ of $ G$, such that $G = A H$ is an almost semi-direct product.
\item  There exists a reductive subgroup  $U \leq H$, with $K  \leq U$,  such that $$ D = H/U$$  is a bounded homogeneous domain and $$U/K$$  is compact with finite fundamental group. 
\item Put $L= A \, U$. Then $L$ is a closed subgroup of $G$ and the map 
\begin{equation*} \label{eq:fundamental}
W =  G/K \, \to \, G/L = H/U = D
\end{equation*}   is a holomorphic
fiber bundle with fiber $L/ K = A \, U/K$.
\end{enumerate}

\smallskip 
We prove now that, if $G$ is unimodular then  \emph{$H$ is a unimodular Lie group}: 
%
For this recall from \cite[Theorem 2.5]{DN} that $A$ is tangent to a 
K\"ahler ideal $\a$ of the K\"ahler algebra which belongs to $W$.  
 (Recall  that the K\"ahler algebra for  $G/K$ is $\g$ together with an 
 alternating two-form $\rho$ which is representing the K\"ahler form on  $W$.)   
Since $K$ intersects $A$ only trivially, the K\"ahler ideal  $\a$ is non-degenerate, that is, the restriction $\rho_{\a}$ of the 
K\"ahler form $\rho$ of $\g$ to $\a$ is non-degenerate. Since $\a$ is abelian 
and $\rho$ is a closed form on $\g$, it follows that $\rho_{\a}$ is invariant by the restriction of 
the adjoint representation of $\h$ (respectively $H$). In particular, this restricted  representation of $H$ on $\a$ is by unimodular maps. Since $G$ is unimodular, it  follows from the semi-direct product  decomposition $G = A H$  that $H$ is unimodular. 

\smallskip
Let $K_{1}$ denote the maximal  compact normal  subgroup of $H$. Then
the group $H' = H / K_{1}$ is unimodular. Moreover $H'$ acts faithfully and transitively on 
$D = H/K =H'/ K'$,  $K' = K/K_{1}$. 
Hence, the bounded domain $D$ has a transitive faithful unimodular group $H'$ of isometries.   
By results  of Hano \cite[Theorem III, IV]{jh}, $H'$ must be semisimple and 
\emph{$D$ is a symmetric bounded domain}.
We also conclude that there exists a semisimple subgroup $S \leq  G$, which is  of non-compact type, 
such that $H = K_{1} S$ is an almost direct product and the homomorphism  $S \to H'$ is a covering with  finite kernel.

\smallskip 
Contractibility of $W$ further implies $U= K$. Therefore, in this case,  the holomorphic bundle in \eqref{eq:fundamental} is of 
the form $$  G/K \to D = H/K \; , $$ 
with fiber $A =\CC^{k}$, and $D$ is
a symmetric bounded domain.

\smallskip Finally the direct product decomposition follows: 
Note also that $K_{1}$ acts faithfully on $\CC^k$ by K\"ahler isometries, and that $S$ acts 
trivially on $A$, since it is of non-compact type. It follows that $S$ is a normal subgroup of 
$G$. Therefore its tangent algebra  must be orthogonal to $A$ with respect to $\rho$.  
Since the K\"ahler algebra $\g$ belonging to $G$ is describing $W$, 
we conclude that there is an orthogonal product decomposition $W = \CC^k \times D $. 
\end{proof}

We also obtain: 
\begin{corollary} Suppose that $W$ is a contractible homogeneous K\"ahler manifold, and that there exists
a discrete uniform subgroup in $\Isom_{\! h}(W)$. Then $W$ is K\"ahler isometric to\/ 
$ \CC^k \times D$, where $D$ is a symmetric bounded domain. 
\end{corollary} 

The following is obtained in  the proof of Proposition \ref{pro:hom_Kaehler}: 
\begin{corollary}  \label{cor:hom_Kaehler}
Assume that $W$ is a homogeneous K\"ahler manifold which admits a transitive unimodular group $G$.
 Then there exists a symmetric  bounded domain $D$ such that
 $W$ is a holomorphic fiber bundle over $D$ with fiber the product of a flat space $\CC^k$ with 
a compact simply connected homogeneous K\"ahler manifold.
Moreover, $\Isom_{\! h}(W)^{0}$ contains a covering group of the identity component of the
holomorphic isometry group of $D$.
\end{corollary}
\begin{proof} 
In fact, in the proof of Proposition \ref{pro:hom_Kaehler} it is established that $D$ is symmetric with
a semisimple transitive group $S$ contained in the quasi normal modification $\bar G$ of $G$.  
It is also clear that $S$ is normal in $\bar G$, and it is the maximal  semisimple 
subgroup of non-compact type in $\bar G$ (in fact, in $\Isom_{\! h}(W)^{0}$), and $S$ is 
covering $\Isom_{\! h}(D)^{0}$. 
\end{proof}

We recall that any symmetric bounded domain $D$ admits an involutive anti-holomorphic isometry: 

\begin{pro}[Isometry group of symmetric bounded domain] \label{order2isometry}
Let $D$ be a symmetric  bounded domain with K\"ahler structure $(\Omega, J)$.
If $D$ is \emph{irreducible} then $$   \Isom(D) =  \Isom_{\! h}^{\pm}(D)  \; . $$
Moreover, for any $D$ there exists an element $\bar \tau\in \Isom(D)$ 
such that
\[
\bar \tau^2=1, \ \bar \tau^*\Omega=-\Omega,\ \bar \tau_*J=-J \bar \tau_*.\]
\end{pro}

For the fact that every isometry of an irreducible bounded symmetric domain is either holomorphic or anti-holomorphic, see e.g.\  \cite[Ch.VIII Ex.B4]{Helgason}. 
For the existence of the anti-holomorphic involution $\bar \tau$, recall first 
that the metric on any symmetric bounded domain $D$ is analytic (see  \cite{Helgason}).
Then the following holds:

\begin{pro} Let\/  $W$ be a simply connected K\"ahler manifold with analytic K\"ahler metric. 
Then there exists an anti-holomorphic involutive isometry $\bar \tau$ of $W$. 
\end{pro} 
\begin{proof} Since $W$ is a complex manifold and the K\"ahler metric is Hermitian with respect to the complex structure, there exists local complex coordinates for $W$ such that  the metric can be written  as
$$ \displaystyle g_0=2\sum_{\al,\be}g_{\al\bar\be}dz^\al d\bar z^{\be} \; , $$ where
$g_{\al\bar\be}$ is a Hermitian matrix, so that $g_{\al\bar\be}= \overline{g_{\be\bar\al}}$.
In particular, the K\"ahler form  $\Omega_0$ is obtained as $\displaystyle
\Omega_0=-2i\sum_{\al,\be}g_{\al\bar\be}dz^\al\wedge d \bar z^{\be}$. 

Let $\tau_0:\CC^n\ra \CC^n$ be the complex conjugation map, that is, 
\[\tau_0(z_1,z_2,\ldots,z_n)=(\bar z_1,\bar z_2,\ldots,\bar z_n).\]
Then $\tau_0$ satisfies 
$\displaystyle \tau_0^*\Omega_0=-\Omega_0,\  \tau_{0*}J_\CC=-J_\CC\tau_{0*}$.
In particular, $\tau_{0}$ defines a local anti-holomorphic isometry of $W$. 

Since $W$ is simply connected, we may use analytic continuation to extend $\tau_{0}$ to 
an analytic map $\bar \tau: W \to W$. By the analyticity assumptions, $\bar \tau$ is an anti-holomorphic 
map and it is preserving the K\"ahler metric. Also it follows $\bar \tau^{2} = {\rm id}_{W}$ by the local rigidity of analytic maps.  Therefore,  $\bar \tau$
is an involutive isometry of $W$.   
\end{proof}

\begin{remark} Note that  the holomorphic isometry group  $\Isom_{\! h}(D)$ has finitely many connected components. Interestingly,  even if $D$ is irreducible $\Isom_{\! h}(D)$ is not necessarily connected 
\cite[Ch. X, Ex. 8]{Helgason}. 
\end{remark}

\section{Locally homogeneous aspherical Sasaki manifolds} 
\label{sec:lchom_sas}

In this section $X$ denotes a regular contractible Sasaki manifold. 

\subsection{Homogeneous Sasaki manifolds for  unimodular groups}
Since $X$ is  regular with Reeb flow isomorphic to the real line, 
Proposition  \ref{pr:noncompact} implies that  the Reeb fibering 
$$\RR \to  X \stackrel{q}{\lra} W$$  
gives rise to  an  exact sequence of groups 
\begin{equation} \label{eq:Sasaki_fib_groups}
 \displaystyle 1\to \RR \to \Psh(X)\stackrel{\phi}\lra\Isom_{\! h}(W)\to  1  \; , 
 \end{equation} 
where  $W$ is the K\"ahler quotient of $X$.

\begin{pro}\label{prop:Sasaki_homuni}
Let $X$ be a homogeneous Sasaki manifold such that its K\"ahler quotient $W$ 
is  contractible.   
Suppose further that 
$X$ admits a connected transitive unimodular 
subgroup $$ G  \, \leq \, \Psh(X) \; . $$  Then the following hold: 
\begin{enumerate}
\item  $W = \displaystyle \CC^k\times  D$   is the K\"ahler product of a flat space with
a  symmetric bounded domain $D$. 
\item  If the Reeb flow of $X$ is isomorphic to $\RR$,  then the 
pullback of $$ \CC^k\rtimes {\rm U}(k) \, \leq  \, \Isom_{\! h}(W)$$   
along the exact sequence  \eqref{eq:Sasaki_fib_groups} 
is a normal subgroup  $$ \cN\rtimes {\rm U}(k) \,  \leq \, \Psh(X) \; ,  $$  
where $\cN$ is a $2k+1$-dimensional Heisenberg Lie group.
\end{enumerate} 
\end{pro} 

\begin{proof} 
As the Reeb flow $T$ is central in $\Psh(X)^0$, the associated  Boothby-Wang homomorphism $\phi$ as in 
\eqref{eq:Sasaki_fib_groups} maps the unimodular group $G$ to 
$$ \bar G = \phi(G) \, \leq \, \Isom_{\! h}(W) \, . $$ 
Since also $\bar G$ is  unimodular and  transitive on the contractible K\"ahler manifold $W$,  
Proposition \ref{pro:hom_Kaehler} states  that  $W = \CC^k \times D$,
 where $D$ is a symmetric bounded domain. This proves (1).
It also follows that  
$$  \Isom_{\! h}(W) =  \left(  \CC^k\rtimes {\rm U}(k) \right) \times \Isom_{\! h}(D) \;. $$ 

We may thus pull  back the factor  $\CC^k\rtimes {\rm U}(k)$ by $\phi$ 
in the exact sequence \eqref{eq:Sasaki_fib_groups}. As pullback  we obtain the  subgroup  
$ \cN\rtimes {\rm U}(k) \, \leq \,  \Psh(X)$, 
where $\cN$ is the preimage of  the translation group $ \CC^k$. 

\smallskip 
 Assuming $T = \RR$, we note that $\cN$ is  a central extension of the Reeb flow 
$\RR$ by the abelian Lie group $\CC^{k}$. We prove now that $\cN$ is a 
$2k+1$-dimensional Heisenberg Lie group by showing that its Lie algebra $\n$ 
has one-dimensional center: 
Since $\cN$ acts faithfully as a transformation group on $X$,
we may identify $\n$ with a subalgebra of pseudo-Hermitian Killing vector fields on $X$. This subalgebra  contains the Reeb field $\cA$ (tangent to the central one-parameter group $T =\RR$) in its center. 
Now, since $\cN$ is the pullback of  $\CC^k$, given any two vector fields $\cX,\cY \in \n$, 
we have $$ [ \cX,\cY] =   \om([ \cX,\cY]) \cA \; . $$ 
 Using Lemma \ref{lem:Lie_Hermitian} below, we observe   $$ q^* \Omega \left(\cX,\cY\right) = 
 d\omega(\cX,\cY)=   \; \om([ \cX,\cY]) \; . $$
Since  $(\CC^{k}, \Omega)$ is K\"ahler, it follows that $d\omega$
defines a non-degenerate two-form on $\n/ \langle \cA \rangle $.  
This shows that the Lie algebra $\n$ has one-dimensional center $\cA$. Therefore the Lie group $\cN$ has one-dimensional center. So $\cN$ is a Heisenberg group of dimension $2k +1$. 
\end{proof}

A vector field on $\cX$ with flow in $\Psh(X)$ will be called a \emph{pseudo-Hermitian vector field}.
The set of pseudo-Hermitian vector fields forms a subalgebra of the Lie algebra of Killing vector fields for the Sasaki metric $g$. 

\begin{lemma} \label{lem:Lie_Hermitian}
Let $\cX, \cY$ be any two pseudo-Hermitian Killing vector fields on the Sasaki manifold $X$. Then 
$$ q^* \Omega \left(\cX,\cY\right) = 
 d\omega(\cX,\cY) = \om([\cX,\cY]) \; . $$
\end{lemma}
\begin{proof} Since the flow of $\cX$ preserves the contact form $\om$, we have  $$L_{\cX} \, \om = 0 \, . $$
(Here, $L_{\cX}$ denotes the Lie derivative with respect to $\cX$.)
That is,  $$ L_{\cX} \, \om(\cZ) - \om([\cX, \cZ]) = 0 , $$ for all vector fields $\cZ$ on $X$.  We compute 
\begin{eqnarray*}  0 & = &  L_{\cX} \,  \om(\cY) - \om([\cX, \cY]) - L_{\cY}\,  \om(\cX) + \om([\cY, \cX]) \\
& = &  L_{\cX} \, \om(\cY) - L_{\cY} \, \om(\cX)  - \om([\cX, \cY])  + \om([\cY, \cX])  \\
& =  & d\omega(\cX,\cY) -  \om([\cX, \cY])  \; .  \end{eqnarray*}    \qedhere  
\end{proof}

Let $X$ be a contractible homogeneous Sasaki manifold with K\"ahler quotient $  W = \CC^{k} \times D$,  where $D$ is a symmetric bounded domain. 
Then 
\begin{equation}\label{Kdecomp}
\Isom_{\! h}(W) = \left(\CC^k\rtimes {\rm U}(k)\right) \times \Isom_{\! h}(D) \; . 
\end{equation}
Note further that $\Isom_{\! h}(D)^{0} = S_{0}$ is the  
identity component of the holomorphic isometry group of a Hermitian symmetric space 
$$D = S_0/H_0$$  of non-compact type.  
In particular, $S_{0}$ is semisimple of non-compact type \cite[Ch.VIII \S7]{Helgason}
and without center.
Therefore  \eqref{Kdecomp} also gives: 
\begin{pro}\label{pro:psh_connected} 
$\Psh(X)$ has finitely many connected components and $$ \Psh(X)/\Psh(X)^0 = \Isom_{\! h}(D)/\Isom_{\! h}(D)^{0} \; . $$
\end{pro}

We add:

\begin{pro}[Sasaki automorphism group]  \label{pro:pshX_dec}
There exists a  semi\-sim\-ple Lie group $S$ of non compact type, 
whose center $\Lambda$ is  infinite cyclic, and a $2k +1$ dimensional 
Heisenberg groups $\cN$,  such that there is an almost direct product decomposition 
$$\Psh(X)^{0}  \;  = \, \left( \cN \rtimes {\rm U}(k) \right) \cdot S   \; .  $$ 
Moreover, the Reeb flow $T$ of $X$ is the center of  $\cN$ and $$T  \cap  S =   \left( \cN \rtimes {\rm U}(k) \right) \cap  S  = \Lambda \;  (\cong \ZZ)  . $$
\end{pro}
\begin{proof} 
For the homogeneous Sasaki manifold $X$, 
the exact sequence of groups  \eqref{eq:Sasaki_fib_groups} associated to the Reeb fibering for $X$ induces  
a central extension 
\begin {equation} \label{eq:Sasaki_h} 
\displaystyle 1\to T  \to \;  \left( \cN \rtimes {\rm U}(k) \right) \cdot S  \; 
 \stackrel{\phi}{\to} \;  (\CC^{k} \rtimes {\rm U}(k)) \times S_0\to 1 \; , 
\end{equation}
where the Reeb flow $T = \RR$ maps to the center of $\cN$. Here $$ S \leq \Psh(X)^0$$  is a semisimple normal subgroup  of non-compact type, which is covering $S_{0}$ under $\phi$. In particular, since $S$  is a normal subgroup of  $\Psh(X)^0$,  it commutes with $ \cN \rtimes {\rm U}(k)$.  (Note also that ${\rm U}(k)$ acts faithfully on $\cN$ and maps to a maximal compact subgroup of $\Aut(\cN)$.) 

The kernel $\Lambda$ of the covering $S \to S_{0}$ 
is $$ \ker \phi \cap  S  = \RR \cap  S =   \left( \cN \rtimes {\rm U}(k) \right) \cap  S \; .  $$ 
Moreover,  $\Lambda$ is the center of $S$, since $S_{0}$ has trivial center. 
We claim that $\Lambda$ is an infinite cyclic discrete subgroup  and, in 
particular,  it is a uniform subgroup  in $T$.  Indeed, in the light of Corollary  \ref{cor:hom_sasaki}, 
there exists a  unique contractible homogeneous  Sasaki manifold $X_{1}$ 
with K\"ahler quotient $\CC^{k}$, and similarly a unique homogeneous Sasaki manifold 
$X_{2}$  with K\"ahler quotient $D$.  Let $T_{i} \leq \Psh(X_{i})$ denote the Reeb flow of $X_{i}$. Then (see Section \ref{sect:join}, Corollary \ref{cor:join}) the  join $X_{1} * X_{2}$ is a homogeneous Sasaki manifold with K\"ahler quotient $\CC^{k} \times D$. According to the above, $\Psh(X_{1}) =    \cN \rtimes {\rm U}(k) $, and by Proposition \ref{pro:S_trans} 
below $\Psh(X_{2})^{0} = T_{2} \cdot S$, where $S$ is a closed semisimple Lie subgroup covering
 $S_{0}$ with infinite cyclic kernel $\Lambda = Z(S)$, $T_{2} \cap S = \Lambda$.  
It follows that $\Psh(X)^{0} =  \Psh(X_{1}) *  \Psh(X_{2})^{0}$ has the claimed properties. 
\end{proof}

\subsection{Application to locally homogeneous Sasaki manifolds}
We consider a compact aspherical Sasaki manifold of the form $$ M =  \Gamma \, \backslash X \; ,  $$
where $X$ is a contractible Sasaki manifold  
and $\Gamma$ is a torsion free discrete subgroup contained in $\Psh(X)$. If 
$X$ is a homogeneous Sasaki manifold then $M$ is called a \emph{locally homogenous 
Sasaki manifold}.  


\begin{theorem}\label{undrphi}
Suppose that $X$ is a contractible homogeneous  Sasaki manifold  
and  that $X$ admits a discrete subgroup of isometries with 
$$ \Gamma \, \backslash X$$  compact. Then: 
\begin{enumerate} 
\item The K\"ahler quotient of $X$ is a K\"ahler product
$$ W = \displaystyle \CC^k\times  D$$ of a flat space $ \CC^k$ with
a  symmetric bounded domain $D$.
\item Let $T$ denote the Reeb flow of $X$.
Then $\Gamma \cap T$ is a discrete uniform subgroup of $T$ (in particular, $\Gamma \cap T$ is  isomorphic to $\ZZ$).
\item Let $\phi: \Psh(X) \to  \Isom_{\! h}(W)$ be the Booothby-Wang homomorphism in \eqref{eq:Sasaki_fib_groups}. Then the  subgroup $$ \phi(\Gamma) \, \leq \, \Isom_{\! h}(W)$$  is discrete and uniform. 
\end{enumerate} 
\end{theorem} 

\smallskip
\begin{corollary} Let $M = X/ \Gamma$ be a compact locally homogeneous Sasaki manifold. Then 
$M$ is a Sasaki manifold with compact Reeb flow $T  = S^{1}$. Moreover, a finite covering space of 
$M$ is a regular Sasaki manifold. 
\end{corollary}

\smallskip
\begin{remark} 
Certain linear	flows on the sphere give rise to \emph{irregular} compact Sasaki manifolds, \cf \/ \cite[Chapters 2, 7]{bg}.
\end{remark}

For the preparation of the proof  of Theorem \ref{undrphi} we shall recall some standard facts about:

\smallskip
\paragraph{\emph{Levi decomposition and uniform lattices}}
In general a \emph{connected} Lie group $\fG$ admits a Levi decomposition 
 $$ \fG= \fR\cdot \fS \; , $$ 
where $\fR$ is the solvable radical of $\fG$ and $\fS$ is a semisimple subgroup. 
Let $\fK$ denote the 
maximal compact and \emph{connected} normal 
subgroup of $\fS$, then put
$\fS_{0} = \fG / (\fR \fK)$. Note that $\fS_{0}$ is semisimple of non-compact type.
We will need the following fact (see \cite[Chapter 4, Theorem 1.7]{VGS}, for example):

\begin{pro} \label{prop:key-correct}
Let $\Gamma$ be a uniform lattice in $\fG$. 
Then the intersection
$(\fR\fK) \cap \Gamma$ is a uniform lattice in 
$\fR \fK$. In particular,  in the associated exact sequence
\begin{equation}\label{startG*}
\begin{CD}
 1@>>> \fR \fK @>>>
 \fG@>\nu >> \fS_0 @>>>1,\;
\end{CD}\end{equation}
the image $\nu(\Gamma)$ is a uniform lattice in  the semisimple Lie group $\fS_0$. 
\end{pro}

Remark in addition the following:  
As the subgroup $\nu(\Gamma) \leq  \fS_0$ is discrete and 
uniform, and since $\fS_{0}$ has no compact normal connected subgroup,
the image of $\nu(\Gamma)$ is a Zariski dense subgroup
in the adjoint form of $\fS_0$  (by Borel's density theorem, \cf \cite{ra}). 
Consider any connected closed subgroup $\cG$ of  $\fS_0$, which contains 
$\nu(\Gamma)$. Then $\cG$ is uniform and Zariski-dense.  This implies that 
$\cG = \fS_{0}$. 

\smallskip 
Now we are ready for the
\begin{proof}[Proof of Theorem \ref{undrphi}]
Note that  $\Gamma_{0} = \Gamma \cap \Psh(X)^0$ is a discrete uniform subgroup of  $\Psh(X)^0$
(compare \cite[Lemma 2.3]{bk}). The existence of a lattice  subgroup implies that  $\Psh(X)^0$ is a \emph{unimodular}  Lie group, 
see e.g.\  \cite[1.9 Remark]{ra}.
By Proposition \ref{prop:Sasaki_homuni}, 
$W = \CC^k \times D$, where $D$ is a symmetric bounded domain  and
$$ \Isom_{\! h}(W) =  \left(  \CC^k\rtimes {\rm U}(k) \right) \times \Isom_{\! h}(D) \; . $$


Since $S_0 = \Isom_{\! h}(D)^{0}$ is semisimple of non-compact type, 
we can apply Proposition \ref{prop:key-correct} to $\Psh(X)^0$, 
to yield that the intersection 
$\Gamma\cap(\cN\rtimes {\rm U}(k))$ is discrete uniform in $\cN\rtimes{\rm U}(k)$.
Then the Auslander-Bieberbach theorem \cite{ab} shows that, a fortiori,  
$\Gamma\cap \cN$ is uniform in $\cN$. As $\RR$ is the center of the Heisenberg group 
$\cN$, $\Gamma\cap \RR$ is also uniform in $\RR$ (\cf \cite[Chapter II]{ra}). 
In particular, in the light of 
\eqref{startG*},  this implies 
that $\phi(\Gamma)$ is a discrete uniform subgroup of $\Isom_{\! h}(W)$.
\end{proof}

\subsection{Sasaki homogeneous spaces over symmetric bounded domains}
We assume now that the 
K\"ahler quotient of $X$ is a symmetric bounded domain $D$. 
Let $$ S_{0} = \Isom_{h}(D)^{0}$$  be the identity component of the group of holomorphic isometries of $D$, and $$ \phi: \, \Psh(X)^{0}   \to S_{0} $$  the Boothby-\-Wang homomorphism. 
Recall that $S_{0}$ is semisimple of non-compact type with trivial center. Moreover, we can write 
$$D= S_{0}/K_{0} \, , $$ 
where $K_{0}$ is a maximal compact subgroup of $S_{0}$. 

\smallskip
%
We prove that $X$ is a Sasaki 
homogeneous space of a semisimple Lie group: 

\begin{pro}  \label{pro:S_trans}
There exists a semisimple closed normal  subgroup $$S  \, \leq \, \Psh(X)^{0} $$  such that the restricted Boothby-Wang  map $$\phi: \, S \to S_{0}$$ is a 
covering with infinite cyclic kernel $\Lambda$, where $\Lambda$ is the center of $S$. 
 In particular, if $T = \RR$ denotes the Reeb flow on $X$,  then 
 $$ \Psh(X)^{0} = S \cdot \RR   ,  \text{ with } S \cap \RR  = \Lambda\;  ( \, \cong \ZZ)  \, .$$ 
Moreover, the subgroup $S$ of\/ $\Psh(X)$ acts transitively on $X$.
\end{pro} 
\begin{proof} 
Put $G = \Psh(X)^0$. Then $G$ satisfies the exact sequence 
$$
\begin{CD} 
1@>>> T =\RR  @>>> G@>{\phi}>> S_0 @>>>1 \;  ,
\end{CD}
$$ %
where the Reeb flow $T$ is a central subgroup of $G$. By the Levi-decomposition theorem, the above exact sequence splits and 
$$ G = T \cdot S \; , $$ where 
$S$ is a covering group of $S_0$ under $\phi$. Note that $S$ is a normal subgroup of $G$, and  $\ker \phi \cap S =T \cap S = Z(S)$ is the center of $S$, and a torsion-free abelian group.  

Assume that $T \cap S  = \{1 \} $.  In particular $S = S_0$ and 
$G = T \times S_0$. Then $K_0$ is also a maximal compact subgroup of $G$.
Choose $x_o \in X$ such that $K_0 \, x_0 = x_0$. Then $S_0 \cdot x_0 = S_0/K_0$ and it follows that 
$$ X = \RR \times S_0/K_0 \; . $$ Moreover, the Boothby-Wang fibering $q: X \to D$ corresponds to the projection onto the second factor. 
Let $\om_0$ be the contact form of the Sasaki structure on $X$.
By Proposition \ref{pr:oductiden} there exists a one-form $\theta$ on $D = S_0/K_0$ 
such that 
$$ \om_0=dt+{q}^*\theta \; . $$ Since $\om_0$ is invariant by $S = S_0$, this implies that  ${q}^*\theta$ is invariant by $S$.  Therefore also $\theta$ is  invariant by $S_0$. In particular, the two form $\Omega=d\theta$
is an $S_0$-invariant exact form. 

We can now apply a classical result of Koszul to $\Omega$ as follows.
Let $\s$ and $\k$ denote the Lie algebras of $S_0$ and $K_0$, respectively. The $S_0$-invariant K\"ahler form $\Omega$ defines a cohomology class in the relative Lie algebra cohomology group $H^2(\s,\k)$. Since $\s$ is unimodular and $\k$ is a reductive subalgebra of $\s$, a  result of Koszul \cite{koszul} asserts that the cohomology ring $H^*(\s,\k)$ satisfies Poincar\'e duality. 
Since $\Omega$ is a non-degenerate two-form, the class $[\Omega] \in H^2(\s,\k) $ is non-zero. This contradicts  $\Omega = d \theta$, for some $S_0$-invariant form $\theta$ on $S_0/K_0$. We conclude that $T \cap S = \{1 \}$ is not possible. 

Therefore, we have that $\ker \phi \cap S = T \cap S = \Lambda$ is isomorphic to $\ZZ^{k}$, $k \geq 1$.   Since $\Lambda$ is the center of $S$, there exists a closed $k$-dimensional subgroup $\fB$ of $S$, $\fB \cong \RR^{k}$, containing $\Lambda$, and $\phi$ maps $\fB$ to a toral subgroup $(S^{1})^{k}$ contained in the center of $K_0$, \cf \cite[Ch. VI, \S 1]{Helgason}.  Let $K$ be the maximal compact subgroup of $S$. We then have 
$$1 + \dim D  = \dim X \geq \,  \dim S/K = k  + \dim S_{0}/K_{0}  = k + \dim D \; .$$ 
Since $k \geq 1$,  we deduce $k=1$ and $X = S/K$.  Hence, $S$ acts transitively on $X$. 
Since $\ZZ \cong \ker \phi \cap S$ is an infinite cyclic discrete subgroup of $T$, 
it also follows that $S$ is a closed subgroup of $\Psh(X)$, see \cite[Theorem B]{goto}.  
\end{proof}

\subsection{Summary on locally homogeneous Sasaki manifolds} \label{sect:summary} 

Most of the above is summarized in Theorem \ref{iso=equal} in the introduction: 

\begin{proof}[Proof of Theorem $\ref{iso=equal}$]
Statement (1) about  the K\"ahler quotient $W= X/T$ is  established in (1) of Theorem  \ref{undrphi}.

\smallskip
We remark next that the Reeb flow $T$ is normal in $\Isom(X)$. 
Indeed, since $X$ is non-compact there can be  only two Killing fields $\{ \cA, -\cA\}$ which are Sasaki compatible with the metric $g$ on $X$ (\cf  \cite{tanno,kashi,tachiyu}). It follows that $\Isom(X) = \Pspm(X)$. 
The properties of the homomorphism $\phi: \Isom(X) \to \Isom_{\! \! h}^{\! \!\pm}(W)$ are 
established in Proposition \ref{pr:noncompact} and Lemma \ref{lem:iso_lift}, proving (2).

\smallskip
Let $\bar \tau: W \to W$ be an anti-holomorphic involution (which exists by Proposition \ref{order2isometry} and Note \ref{note:isomN}). Then by Lemma  \ref{lem:iso_lift},  there exists an anti pseudo-Hermitian and involutive lift $\tau: X \to X$. Now (3)  follows.  

\smallskip
Since  $\Isom(X) =  \Pspm(X)$, we deduce that $\Isom(X)^{0} =  \Psh(X)^{0}$. 
Therefore part (4) is a consequence of Proposition \ref{pro:pshX_dec}.  

\smallskip
Finally, let $\Gamma\big\backslash G/H$ be a  locally homogeneous aspherical
Sasaki manifold, and $X= G/H$. Then there is the exact sequence :
\begin{equation*}\begin{CD}
1@>>> \Gamma@>>> N_{\Isom\,(X)}(\Gamma)@>>> 
\Isom\,(\Gamma\backslash X)@>>>1.
\end{CD}\end{equation*} Thus the claim (5) (stated below of Theorem $\ref{iso=equal}$)
follows from (3).
\end{proof}


\smallskip 
\begin{proof}[Proof of  Corollary $\ref{cor:quasi_reg}$]
Assume that $\Gamma \backslash X$ is compact. As usual $T$ denotes the Reeb flow for $X$. Then by (2) of Theorem \ref{undrphi}, $\Gamma \cap T$ is an infinite cyclic group $\ZZ$. 
Put $$ S^1 = T \big/ (\Gamma\cap T) \; . $$
According to (3) of Theorem \ref{undrphi}, taking the quotient of $\Gamma \backslash X$ by $S^1$, this  
induces an $S^1$-bundle over a compact locally homogeneous aspherical K\"ahler orbifold of the form:
\begin{equation*}
\begin{CD}
S^1@>>>\Gamma\big\backslash  G/H@>>>\phi(\Gamma)\big\backslash W \; . 
\end{CD}\end{equation*}
Here $S^1$ induces the Reeb field of $\Gamma \backslash X$. 
This $S^1$-bundle is usually referred to as a  \emph{Seifert fibering} (\cf \cite{lr}).
In particular, since $\Isom_{\! h}(W)$ is a linear Lie group, we can choose a torsionfree finite index normal subgroup
of $\phi(\Gamma)$. Therefore, some finite cover of $\Gamma\big\backslash G/H$ becomes  a regular Sasaki manifold. 
This proves Corollary \ref{cor:quasi_reg}.
\end{proof}

\subsubsection{Solvable fundamental group} \label{sect:solvable_case}
Note (see \cite[Theorem 0.2]{bc})  that every compact aspherical K\"ahler manifold $N$ with virtually solvable fundamental group $\Gamma$ is biholomorphic to a flat K\"ahler manifold $\CC^{k} /\,\Gamma$ for some embedding of $\Gamma$ into $\CC^{k} \rtimes \mathrm{U}(k)$ as a discrete uniform subgroup. This shows, that the K\"ahler manifold $N$, in fact, admits a locally homogeneous (and flat) K\"ahler structure, with respect to its original complex structure. Based on this result we prove now the following (which is also implying Corollary \ref{cor:solvable} in the introduction):

\begin{pro}  \label{pro:Sasaki_def}
Let $M$ be a regular  compact aspherical Sasaki manifold with 
virtually solvable fundamental group. Then the given Sasaki structure on $M$ can be deformed 
(via regular Sasaki structures) to a locally homogeneous regular Sasaki structure. 
\end{pro}
\begin{proof} By the Boothby-Wang fibration result for compact regular Sasaki manifolds \cite{bw}, $M$ is a principal  circle bundle $S^{1} \to M \stackrel{q}{\to} N$ over a compact K\"ahler manifold $(N, \{\Omega, J\})$. Moreover, the K\"ahler class  $[ \Omega ] \in H^{2}(N, \RR)$ is integral and it is the image of the characteristic class $c(q) \in H^{2}(N, \ZZ)$ of the bundle.  Let $\pi$  denote the fundamental group of $M$. 
On the level of fundamental group 
the circle bundle gives  rise to a central group extension
 \begin{equation} \label{eq:ext_Gamma} 1 \to \ZZ \to \pi \to \Gamma \to 1 \end{equation}
such that its extension class in $H^{2}(\pi,\ZZ) \cong H^2(N,\ZZ)$ also maps to $[ \Omega ]$. 
(In this context, the Seifert circle bundle $M$ is said to realize the group extension \eqref{eq:ext_Gamma}.) 

Since $\Gamma$ is virtually solvable there exists a biholomorphic diffeomorphism $\Phi:   \CC^{k}/\, \Gamma \to (N, J) $. Since $\Lambda = \Gamma \cap \CC^{k}$ is a finite index subgroup of $\Gamma$ and a lattice in $\CC^{k}$, we can construct an embedding $\pi \to \cN \rtimes \mathrm{U}(k)$ such that $\Delta = \pi \cap \cN$ is a uniform discrete subgroup in  $\cN$, and the embedding induces a compatible map of exact sequences from \eqref{eq:ext_Gamma} to the defining exact sequence of the group $\Psh(\cN)$ which is of the form   
\begin{equation} 1 \to \RR \to \Psh(\cN) = \cN \rtimes \mathrm{U}(k) \to \CC^{k} \rtimes \mathrm{U}(k) \; . 
\end{equation} 
This constructs 
a locally homogeneous Sasaki structure on the quotient manifold $\cN \big/ \pi $ with K\"ahler quotient  $\CC^{k}/\, \Gamma$ and another Seifert circle bundle $S^{1} \to  \cN \big/ \pi  \to  \CC^{k}/\, \Gamma $ which realizes the exact sequence \eqref{eq:ext_Gamma}. 

By the rigidity for Seifert fiberings (\cf \cite{lr}) there exists an isomorphism of circle 
bundles $\Psi:    \cN \big/ \pi  \to M$ which induces the biholomorphic map 
$\Phi$ on the base spaces. This shows that the principal circle bundle $q: M \to N$ admits a compatible \emph{locally homogeneous} Sasaki structure $(M, \{ \omega', J' \})$ which is modeled on $\cN$ and has K\"ahler quotient $(N, \{ \Omega', J \})$, where $\Omega'$ is a flat (locally constant ) K\"ahler form on $(N,J)$. 

Moreover, by the above remarks  $[\Omega'] = [\Omega] \in H^{2}(N, \RR)$. Hence,  we can write $\Omega' =  \Omega + \theta$, where $\theta = J \partial \bar \partial \varphi$, for some potential function $\varphi: N \to \RR$.  (See  \cite[\S 11.C]{besse} for parametrization  of  
the space of K\"ahler forms on the complex manifold $(N,J)$, which is realizing the given 
K\"ahler class $[\Omega]$.) We may thus choose  a continuous path of cohomologous K\"ahler forms $\Omega_{t} =  \Omega + \theta_{t}$, $\theta_{0} =0$ and $\theta_1 = \theta$,  that is joining $\Omega$ and $\Omega'$, e.g.\ $\theta_{t} = t \theta$.   Since the forms $\theta_t$ are exact, we may lift to a continuous path of one-forms 
$\tau_{t} \in \Omega^{1}(N)$ which is satisfying $d \tau_{t} = \theta_{t}$. 

Finally, let $\om$ denote the connection form on the given circle bundle 
$q: M \to N$, which defines the given regular Sasaki structure with K\"ahler quotient $(N, \{ \Omega, J \})$.  
Then it follows that the connection forms $\om_{t} = \omega + q^{*} \theta_{t}$ give rise to  a continuous family of regular Sasaki structures $\{ \om_{t}, J_{t} \}$ compatible with the circle bundle $q$ and with K\"ahler quotients $(N, \{ \Omega_{t}, J \})$. It follows that $(M, \{ \om_{1}, J_{1} \})$ and $(M, \{ \omega', J'\})$ are Sasaki 
structures over the K\"ahler quotient  $(N, \{ \Omega', J \})$, with  $(M, \{ \omega', J'\})$ 
 being locally homogeneous. 
 The universal covering space $X$ of $M$ inherits the structure of  
 a principal $\RR$-bundle over the unitary space $\CC^{k}$ with induced Sasaki structures from 
 $\{ \om_{1}, J_{1} \}$ and $\{\om', J'\}$. The latter one being homogeneous with group $\Psh(X) \cong \Psh(\cN)$. 
 Proposition \ref{pr:gauge} shows that the induced structures on $X$ are equivalent Sasaki structures.
In particular, both are homogeneous Sasaki structures. 
This shows that $(M, \{ \om_{1}, J_{1} \})$ is a locally homogeneous Sasaki structure. 
\end{proof}

\section{Classifications of homogeneous Sasaki spaces} \label{sec:classifications} 
In this section we tackle the classification problems for (1) aspherical Sasaki homogeneous spaces of semisimple Lie groups and (2) contractible Sasaki Lie groups up to equivalence.


\subsection{Homogeneous Sasaki spaces of semisimple Lie groups} \label{sec:homogeneous_ss}
We call a connected semisimple Lie group $ S_{0}$  of non-compact type a Lie group of  \emph{Hermitian type} if
it is the identity component of the holomorpic isometry group of a 
symmetric bounded domain $D = S_0 / K_0$.

\begin{theorem}  \label{thm:semisimple_hom}
Let $X$ be a  contractible Sasaki homogeneous space of a semisimple Lie group 
$$  S \, \leq \, \Psh(X)^{0} \; . $$ 
Then $S$ has infinite cyclic center 
and 
$$X = S/K \; , $$
where $K$ is a maximal compact subgroup of $ S$. 
Moreover, $S$ is covering a Lie group $S_{0}$ of Hermitian type, such that:
\begin{enumerate} 
\item The K\"ahler quotient of $X$ is the symmetric bounded domain $$ D = S_0/ K_0 \; .$$ 
\item There exists a simply connected one parameter subgroup $\fA \leq S$, contained  in the centralizer of $K$, whose action on 
$X$ induces the Reeb flow,  and 
 the Boothby-Wang fibration for $X$ is 
 of the form 
$$    \fA  \,  \to  \; X = S/ K \;  \to \; D = X/\fA  = S \big/ K \fA \; . $$
\item If\/ $T$ denotes the Reeb flow for $X$ then $$ \Psh(X)^0 \; = \; T \cdot S \; , $$ and  $T \cap S =\Lambda$ is the center of $S$.
\end{enumerate} 
\end{theorem}
\begin{proof} Given a Sasaki  metric on $X$ which is homogeneous for the semisimple group $G= S$, 
the Boothby-Wang presentation  of the  K\"ahler quotient $W$ must be of type  \eqref{eq:sasakifibKII} (\cf Section \ref{sec:presentations}).
That is, it is  of the form $$ {\sf A} \,  \to S/K \,  \to \;  W = S \big/ K {\sf A} = S_0/K_0 \, . $$
Moreover, ${\sf A} \leq S$  is a one-parameter subgroup   centralizing $K$, and $S \to S_0$ is
a  covering homomorphism with infinite cyclic kernel $\Lambda$. In particular, $W = X/\fA$ is contractible and it is a faithful K\"ahler homogeneous space 
of the semisimple Lie group $S_{0} = S/\Lambda$. By Proposition \ref{pro:hom_Kaehler},  $W = D$ is K\"ahler isometric to a
bounded symmetric domain $D$, and $S_{0}$ is the identity component of the isometry group of $D$. In particular, $S_{0}$ is a semisimple Lie group of Hermitian type, and $D= S_0/K_0$, where $K_0$ is maximal compact in $S_0$. Moreover,  $S_{0}$ has trivial center. Therefore,  the center of 
$S$ coincides with  the kernel $\Lambda$ of $S \to S_{0}$, which is infinite cyclic.  
\end{proof} 

The following complements Theorem \ref{thm:semisimple_hom} by showing that any symmetric bounded domain $D$ is the K\"ahler quotient of a contractible Sasaki homogeneous space for 
a semisimple Lie group $S$: 

\begin{theorem}  \label{thm:semisimple_hom2}
For any symmetric bounded domain $D = S_{0}/K_{0}$, there exists a unique semisimple 
Lie group $S$ with infinite cyclic center, which is covering $S_{0}$ and gives rise to a contractible 
Sasaki homogeneous space 
$$ X_{S} = S/K \, $$ 
with K\"ahler quotient $D$.  
\end{theorem} 
\begin{proof} Let $X$ be the \emph{unique} contractible  Sasaki homogeneous space over $D$, which exists by Corollary \ref{cor:hom_sasaki}.  By Proposition \ref{pro:S_trans}, the maximal  normal semisimple subgroup $S \leq \Psh(X)$ is acting transitively on  $X$, and it is covering $S_{0}$ with infinite cyclic kernel. By Theorem \ref{thm:semisimple_hom} (3), any transitive semisimple Lie subgroup of $\Psh(X)$ coincides with $S$. 
\end{proof}

Dividing out the center of $S$ gives rise to a homogeneous Sasaki manifold  $$  Y_0 = X \big/ \Lambda $$ whose Reeb
flow is a circle group. This shows that any  semisimple Lie group of Hermitian type is 
actually acting transitively on an associated Sasaki homogeneous space: 

\begin{corollary} \label{cor:semisimple_hom} For any semisimple Lie group $S_0$ of Hermitian type, there exists a unique 
Sasaki homogeneous space  
$$ Y_0 = S_{0}/K_{1}  $$ 
with 
K\"ahler quotient $D = S_0/K_0$. In this situation, the following hold: 
\begin{enumerate}
\item  There exists a circle group $\bar \fA \leq K_{0}$ such that
$K_{0} =  \bar \fA  \times K_{1}$ is a maximal compact subgroup of $S_{0}$. 
\item The Reeb flow $T_0$ for the Sasaki space $Y_0$  is isomorphic to a circle group $S^1$ and $$ \Psh(Y_{0}) \;  \cong \;  S_{0} \times T_0 \; . $$  
\end{enumerate}
Moreover, every Sasaki homogeneous space with K\"ahler quotient $D$ is a covering space of\/ $Y_{0}$.  
\end{corollary}
\begin{proof} Consider the unique contractible Sasaki homogeneous space $X$ over $D= S_0/K_0$. Then $X = S/K$, where the semisimple group $S$ admits a covering $S \to S_0$ with kernel $\Lambda$, the center of $S$. By part (3) of Theorem \ref{thm:semisimple_hom}, $\Lambda$ is contained in the Reeb flow $T$ for $X$. Therefore $\Lambda$ is acting properly discontinously and freely on $X$, and $Y_0 = X/\Lambda$ is a homogeneous Sasaki space for $S_0$, which has Reeb flow $T_0 = T/\Lambda = S^1$. Since $X$ is the unique simply connected Sasaki homogeneous space with K\"ahler quotient $D$, any Sasaki homogeneous space over $D$ is a quotient space of $X$, hence such a homogeneous space is covering $Y_0$.
\end{proof}

\subsection{Sasaki Lie groups} 
A  Lie group $G$ is said to be a \emph{Sasaki group}
if $G$ admits a left-invariant Sasaki structure (respectively, standard pseudo-Hermitian structure) 
$\{ \om, J \}$.  Accordingly, any simply transitive pseudo-Hermitian 
action of $G$ on a Sasaki space $X$ determines a unique 
left-invariant Sasaki structure on $G$ up to isomorphism. Two Sasaki Lie groups $G$ and $G'$ are considered to be \emph{equivalent Sasaki Lie groups} if there exists an isomorphism $G \to G'$ which is a pseudo-Hermitian isometry. {Two Sasaki Lie groups acting on $X$ are equivalent if and only if they are conjugate subgroups of $\Psh(X)$.}

\subsubsection{Sasaki Heisenberg groups $\cN$}
Let $X$ be the contractible homogeneous Sasaki manifold over $\CC^{k}$. That is, 
we assume that the Reeb fibering for $X$  is of the form 
\begin{equation*} \label{eq:Sasaki_fib_spaces_N}
\RR \;  \to \,  X \, \stackrel{q}{\to} \;  \CC^{k}\; . 
\end{equation*}
By (2)  of Proposition \ref{prop:Sasaki_homuni},  the $2k$-dimensional Heisenberg group 
$$ \cN \,  \leq \, \Psh(X) $$ 
is the preimage of the translation subgroup $\CC^{k} \leq \Isom_{\! h}(\CC^{k})$.   
Moreover, $\cN$ acts simply transitively on $X$. Therefore, we 
get that $\cN$  is a Sasaki Lie group, which as a space is isometric to $X$ by a pseudo-Hermitian isometry. 
We also deduce that   $$ \Psh(\cN)= \Psh(X) = \cN\rtimes {\rm U}(k)$$  
is a connected Lie group. (Compare also 
 \cite{kt}, for example.)

\smallskip 
We describe the standard Sasaki structure on $\cN$ more explicitly as follows: 

\begin{exam}[Sasaki Heisenberg group $\cN$] \label{ex:SasakiN} 
Let $\cN = \RR\times \CC^k$ be the $2 k+1$-dimensional Heisenberg group ($k \geq 0$). 
We write the group law on  $\cN$ as 
\begin{equation} \label{eq:glN} \displaystyle (t,z)(s,w)=(t+s-{\rm Im}({}^t\bar zw), z+w). 
\end{equation}
The standard pseudo-Hermitian structure $\{\om_0, J\}$ on
$\cN$ is given by the left-invariant contact one-form 
$$ \displaystyle \om_0=dt+{\rm Im}({}^t\bar zdz)  \; ,$$
together with a left-invariant complex stucture $J$, defined on ${\rm ker}\, \om_0$ by the relation
$$ {\sf q}_{*}\circ J=J_\CC\circ {\sf q}_{*} \; . $$  
Here  $J_\CC$ denotes the standard complex structure of  $\CC^n$, 
$q: \cN \to \CC^{n}$ is the natural projection. 
Then $g_0=\om_0\cdot \om_0+d\om_0\circ J$ 
is the positive definite Sasaki metric on $\cN$. 
\end{exam} 

We calculate the isometry group of the Sasaki group $\cN$ explicitly as follows:
\begin{note}[Isometry group of $\cN$] 
\label{note:isomN} 
Consider the semidirect product group 
$$ \mathop{\Sim}({\mathcal N}) =  {\mathcal N}\rtimes ({\rm U}(k)\times \RR\sp
{+}) \; ,$$  
where 
$ {\rm U}(k)\times \RR\sp{+}$ is contained in ${\rm Aut}(\cN)$. 
The action of  $(A,\lambda)\in {\rm U}(k)\times \RR\sp{+}$ on $\mathcal N$ is 
given  by:
\begin{equation*}
(A,\lam) \,  (t,z) = (\lam^{2} t,\ \lam A z) \; .
\end{equation*}
It follows that $(A,\lam)^{*} \,  \omega_{0} =  \, \lambda^{2} \, \omega_{0}$. In particular, ${\rm U}(k)$ acts by strict contact transformations and holomorphically on the standard pseudo-Hermitian manifold 
$(\cN,\{\omega_{0},J\})$. That is, ${\rm U}(k)$ is a subgroup of $\Psh(\cN)$.  
Next define $\tau \in \Aut(\cN)$ by \begin{equation} \label{eq:tau}  \tau(t,z)=(-t,\bar z) \; . \end{equation} 
Then $\tau^*\om_0=-\om_0$ 
and $J\circ\tau_*=-\tau_*\circ J$. Thus $$ \langle \tau\rangle =\ZZ_2 \, \leq \, \Pspm(\cN) $$ 
is contained in the isometry group of the Sasaki metric $g_{0}$, but does not belong to $\Psh(\cN)$.
Observe further that $$ {\rm U}(k)\rtimes \langle \tau\rangle$$  is a maximal compact subgroup of the automorphism group $\Aut(\cN)$.  
We deduce: 
\begin{equation} \label{eq:isoN}
 \Psh(\cN) =  {\mathcal N}\rtimes {\rm U}(k) \text{ and }  \Isom(\cN) =  \Pspm(\cN) = \Psh(\cN) \rtimes \ZZ_2  \; . 
\end{equation}
(Recall also that by \cite{wilson},  the isometry group  of any  left-invariant Riemannian metric on $\cN$ is contained in the group of affine transformations $\cN \rtimes \Aut(\cN)$.)
\end{note}

We prove now that the Sasaki Lie group structure on the Heisenberg Lie group $\cN$ is essentially unique: 

\begin{pro} \label{pro:heisenberg_unique}
Up to isomorphism of Sasaki Lie groups,  there is a unique Sasaki structure on the 
Heisenberg Lie group  $\cN$. 
\end{pro} 
\begin{proof} Suppose $(\cN, \{\om, J\})$ is a Sasaki Lie group of dimension $2k+1$.  
In particular, the space 
$X = \cN$  is a contractible homogeneous Sasaki manifold, on which the group $\cN$ acts simply transitively. 
Via the Boothby-Wang homomorphism,
 $\cN$ also acts  transitively on the K\"ahler quotient $W = X \big/ \RR$. Since $\cN$ is nilpotent, 
 $W$ must be flat (for example by \cite{jh}). So $W$ is K\"ahler isometric to $\CC^{k}$. 

Then, as follows from Section \ref{sec:presentations}, we must be in the situation Case II, 
where the Reeb flow $T$ coincides with the center of $\cN$. 
Therefore, the Boothby-Wang homomorphism for $X$ maps $\cN$ to an abelian simply transitive 
subgroup $\bar \cN$ of isometries of unitary space $\CC^{k}$. We conclude that this image group $\bar \cN$ is actually the translation group $\CC^{k}$, which is the unique abelian simply transitive subgroup 
of  $\CC^{k} \rtimes {\rm U}(k)$. 
Therefore, $\cN$ is the normal subgroup of $\Psh(X)$ which is the preimage of $\CC^{k}$. 
Now the Sasaki manifold $X$ is determined uniquely by its K\"ahler quotient $\CC^{k}$  (\cf Corollary \ref{cor:hom_sasaki}) 
up to a pseudo-Hermitian isometry. {By Proposition \ref{pro:pshX_dec}, $\cN$ is the nilradical of $\Psh(X)$. Therefore, it is uniquely determined and characteristic in $\Psh(X)$.} Since the space $X$ is determined 
uniquely by $\CC^{k}$, this constructs the left-invariant structure on $\cN$  uniquely  up to a pseudo-Hermitian isomorphism of Sasaki Lie groups. 
\end{proof}

\subsubsection{Heisenberg modifications $\cN(k,l)$} \label{sec:Hmod}
We construct a family of simply connected Sasaki Lie groups which are  modifications of 
the Heisenberg Sasaki group $\cN$ introduced in Example \ref{ex:SasakiN}. (Compare also \cite{ahk}).

\smallskip 
\paragraph{\em Flat K\"ahler Lie groups} 
For this, let  $\rho:\CC^l\to {\rm U}(k)$ be a non-trivial homomorphism $(k+l=n)$.
Then the semidirect product $\displaystyle\CC^k{\rtimes}_{\rho}^{}\CC^l$ embeds in 
an obvious manner as a simply transitive subgroup 
$$  \CC(k,l)\leq \CC^n\rtimes {\rm U}(n)   $$ 
of the holomorphic isometry group of flat unitary space $\CC^{n}$.
Thus $\CC(k,l)$ is a \emph{flat K\"ahler group}, since it is  acting simply transitively by holomorphic isometries on  $\CC^n$. (In fact,  every flat K\"ahler Lie group contained in 
$\CC^n\rtimes {\rm U}(n)$ is conjugate to some $\CC(k,l)$, 
compare \cite[Theorem II]{jh}.) Note also that $k \geq 1$ and that the standard K\"ahler form of $\CC^{n}$ is non-degenerate on $\CC^k$. 

\smallskip 
\paragraph{\em Heisenberg modifications}
Let $X$ be \emph{the} unique contractible Sasaki homogeneous space over $\CC^{n}$. Consider the pull-back $\cN(k,l)$ of $\CC(k,l)$ in the central extension which is defining $\Psh(X)$ according to Proposition \ref{pro:pshX_dec}: 
\smallskip 
\begin{equation} \label{eq:cnkl}
\begin{CD}
1@>>> \RR@>>>  \Psh(X) = \cN\rtimes  {\rm U}(n) @>p>> \CC^n\rtimes  {\rm U}(n) @>>>1\\
@. ||@. \cup@. \cup@. \\
1@>>> \RR@>>> \cN(k,l)@>p>> \CC(k,l)@>>>1 \; .
\end{CD}
\end{equation} 

\smallskip 
In particular, such  $\cN(k,l)$ is a simply connected solvable Lie group (where $\cN(n,0)=\cN$ is nilpotent). 
Moreover, 
$$   \cN(k,l) \,  \leq \,  \Psh(\cN) = \cN \rtimes  {\rm U}(n)  $$ acts simply transitively and by pseudo-Hermitian transformations on the Sasaki manifold $X = \cN$. 
From this action, $\cN(k,l)$ inherits a natural  
structure as a Sasaki Lie group.  

\begin{definition}\label{modicN}
Any  Sasaki group of the form  $\cN(k,l) \leq \Psh(\cN)$ as above is said to be a Heisenberg modi\-fication (of type $(k,l)$).
\end{definition}

\begin{remark} By definition, the groups $\cN(k,l)$ are defined as preimage of K\"ahler Lie groups. The proof of  Proposition \ref{pro:heisenberg_unique} shows that the classification of groups $\cN(k,l)$ up to isomorphism of Sasaki Lie groups amounts \emph{exactly} to the classification of 
K\"ahler Lie groups $\CC(k,l)$ up to isomorphism. {For a discussion of the structure of flat K\"ahler Lie groups,  
see for example \cite{bdf} and~\cite{jh}.}  
\end{remark} 

Also we note: 
\begin{lemma}  \label{lem:cklback} 
Let $X$ be any contractible Saskaki manifold over a  homogeneous K\"ahler manifold $W$. 
If\/ $\CC^{n}$ is the maximal flat factor of\/  $W$ then the preimage $\tilde \cN$ in $\Psh(X)$ of a subgroup 
$$ \CC(k,l) \leq \CC^n\rtimes {\rm U}(n) $$ 
under the homomorphism $\phi$  in the sequence \eqref{eq:Sasaki_fib_groups} is $\cN(k,l)$. 
\end{lemma}
\begin{proof} By Proposition \ref{prop:Sasaki_homuni} (2), 
the pullback of $ \CC^n\rtimes  {\rm U}(n)  \leq  \Isom_{\! h}(W)$   
to the group $\Psh(X)^0$ along the exact sequence  \eqref{eq:Sasaki_fib_groups} 
is $ \cN\rtimes  {\rm U}(n) $. Therefore, the pullback $\tilde \cN$ of  $\CC(k,l)$ satisfies the defining exact sequence \eqref{eq:cnkl} above.  
So 
 $\tilde \cN = \cN(k,l)$ 
\end{proof} 

\smallskip 
\begin{proof}[{\bf Proof of Theorem $\ref{sasakimodi}$}]
Let $G$ be a  contractible unimodular Sasaki group.  
As follows from Theorem \ref{SasakiKprin},  there exists a one-parameter subgroup 
 $$ \fA \leq G $$  such that 
$W = G/\fA$ is a homogeneous K\"ahler manifold for $G$.  

If $\fA$ is a normal subgroup in $G$ (\cf case \eqref{eq:sasakifibKI} of Section \ref{sec:presentations}), then 
$$ \bar G = G/\fA$$  is a K\"ahler group acting simply transitively on $W$, and $\fA$ is, a fortiori, central in $G$. 
Hence, as $G$ is unimodular, so is $ \bar G = G/\fA $.  Therefore, 
Hano's theorem \cite[Theorem II]{jh} implies that $W = \CC^{n}$ is a flat K\"ahler space and 
that $$ \bar G =  \CC(k,l) \;  \leq \; \CC^n\rtimes {\rm U}(n)$$   is a meta-abelian K\"ahler group. 
Since  $G$ is simply connected, the Reeb flow $T$ for the Sasaki manifold $G$ is isomorphic to $\RR$. 
By Lemma \ref{lem:cklback}, this implies that, as a Sasaki Lie group, $$ G = \cN(k,l) \; , $$ 
for some $k,l$,  with $k+l=n$. 

\smallskip We may assume now that $\fA$ is not normal in $G$. Thís is 
case \eqref{eq:sasakifibKII} in Section \ref{sec:presentations}.  The presentation \eqref{eq:sasakifibKII} 
for $W$ is then a fiber bundle of the form 
\begin{equation*}  
\begin{CD} 
1@>>> \fA @>>> G@>{\mathsf{q}}>>  W = \bar G /\bar \fA  @>>>1 \;  , 
\end{CD}
\end{equation*}
where $$ \bar G =  G/\mathcal{Z} \; , $$ with $\mathcal{Z}$ a discrete subgroup in the center of $G$, and $\bar G$ is acting faithfully on $W$. 
In particular, $\bar G$ is a unimodular group of K\"ahler isometries acting transitively on the contractible K\"ahler manifold $W$. 
By Proposition \ref{prop:Sasaki_homuni}, $$ W = \CC^{n} \times  D , $$ where $D = S_0/K_0$ is a symmetric bounded domain. Therefore, 
$$ \bar G \leq\left(\CC^{{n}} \rtimes {\mathrm U}(n)\right)  \times S_{0} , $$ where $S_{0}  = \Isom(D)^{0}$ is a semisimple Lie group of hermitian type. Projecting $\bar G$ to $S_{0}$, with kernel $$ L = \bar G \cap \left(\CC^{{n}} \rtimes {\mathrm U}(n)\right) \; , $$ the image of $\bar G$ in $S_{0}$ is a unimodular group, acting transitively on $D $. By Hano's theorem, the image of $\bar G$ must be semisimple. Therefore it is all of $S_{0}$.  From the Levi-decomposition theorem, we infer that 
$$\bar G = L \cdot S_{0}$$ is an almost semi-direct product.
Therefore,  $$ \dim \bar G = \dim L + \dim S_{0}  =  \dim W + 1 \, \leq \,\dim L +  (\dim S_{0} - \dim K_{0 }) +1 \; . $$ 
This implies  $ \dim K_{0} \leq 1$. 

\smallskip
Suppose first that  $D$ is non-trivial. Then we have that $$ D = \HH^1_\CC$$ is biholomorphic to the hyperbolic plane, $S_{0}$ is isomorphic to $\PSL(2,\RR)$ and $K_{0}$ is a circle group. It follows that the above kernel $L$ of the projection $\bar G \to S$ acts simply transitively on  the factor $\CC^{n}$ of $W$. Hence, $L$ is a flat K\"ahler Lie group, and therefore $ L= \CC(k,l)$. 
By Lemma \ref{lem:cklback}, the preimage of $L$ in $\Psh(G)$ under the Boothby-Wang homomorphism 
is a subgroup $$ \cN(k,l) \leq \Psh(G) \; ,  $$  which contains the Reeb flow $T$ in its center. Since  $G$ is covering $\bar G$,  $G$ contains a subgroup $$ \tilde L = \cN(k,l) \cap G$$ as a covering group of $L$. Therefore,  
$$  \cN(k,l)= T \cdot \tilde L $$ 
is an almost semi-direct product. This is contradicting the fact
that the extension class of the exact sequence in the bottom row of \eqref{eq:cnkl} is non-trivial 
(compare Lemma \ref{lem:Lie_Hermitian}). The contradiction implies that the factor $\CC^{n}$ must be trivial. Thus, $$W = D = \HH^1_\CC$$ is a K\"ahler manifold of constant negative curvature. Hence,  $$ G \, = \, \widetilde {\SL(2,\RR)} $$ is the universal covering group of $S_0 = \PSL(2,\RR)$ with a standard Sasaki structure over $\HH^1_\CC$. 

\smallskip 
It remains to exclude the case that  $D$ is trivial. Suppose we have  $$W  =  \CC^{n} = \bar G/ \bar \fA \; . $$ Since any reductive subgroup of isometries on $\CC^{n}$ has a fixed point, the circle group $\bar \fA$ must be  a maximal reductive subgroup of $\bar G$. We deduce that $\bar G$ is a solvable Lie group with maximal compact subgroup $\bar \fA$. Thus there exists a simply connected solvable normal subgroup $\bar G_{0}$ such that $$ \bar G =  \bar G_{0} \rtimes \bar \fA$$ (see e.g.\ \cite[Lemma 2.1]{bk}). It follows that $\bar G_{0} =  
\CC(k,l)$ is a flat K\"ahler Lie group. As above, this implies that $$  \cN(k, l) = T \cdot \left(G \cap \cN(k,l)\right)$$ is an almost semi-direct product, which is not possible. Hence, the case $D$ is trivial cannot occur, unless $\fA$ is normal in $G$. 
\end{proof}


\section{Examples}\label{sec:2}
We give further explicit  examples of locally homogeneous aspherical Sasaki manifolds.

\subsection{Sasaki manifolds modeled over complex hyperbolic spaces} \label{sect:hnexample}
The complex hyperbolic space is 
described as the homogeneous manifold
$$ \displaystyle \HH^n_\CC={\rm PU}(n,1) \big/ {\rm U}(n) =
{\rm SU}(n,1) \big/\; { {\rm S}  \left( {\rm U}(n)\times{\rm U}(1) \right) }  \; . $$  
Consider the following diagram 
of principal bundle fiberings:
\begin{equation*}\label{CeterSasaki}\begin{CD}
\RR =  \widetilde{{\rm U}(1)}@>>>  X = \widetilde{ {\rm U}(n,1)}\big/ \,  { \widetilde{{\rm U}(n)}} @>\tilde{P}>>   \displaystyle \HH^n_\CC= {\rm PU}(n,1)/{\rm U}(n)\\
@VV{\big/\ZZ}V @VV{\big/\ZZ}V ||@.\\
S^1 = {\rm U}(1) @>>> Y =  {\rm U}(n,1)\big/\, {\rm U}(n)@>P>>   \displaystyle \HH^n_\CC= {\rm PU}(n,1)/{\rm U}(n)\\
\end{CD} \; \; ,  \vspace{1ex} 
\end{equation*} 
where the inclusions of  $\widetilde{{\rm U}(n)}$, $\widetilde{{\rm U}(1)}$ arise
from the standard  embedding $${\rm U}(n) \times {\rm U}(1) \to  {\rm U}(n,1) \;  . $$ 

\begin{remark}
Denoting  with $  \pi:    \; \widetilde{ {\rm U}(n,1)} \to   {\rm U}(n,1)$  
the universal covering group of $ {\rm U}(n,1)$, we declare connected subgroups 
$$\widetilde{{\rm U}(n)} = \, \pi^{-1}({\rm U}(n))^{{0}}  \,  ,   \;  {\widetilde{{\rm SU}(n,1)} = \, \pi^{-1}({\rm SU}(n,1))^{0} } \; . $$ 
Then  $\widetilde{ {\rm U}(n)}$ is  a universal covering group for\/ ${\rm U}(n)$,
and the  kernel $\cZ \; (\cong \ZZ) $  of the latter covering is contained in 
the center of the group $$ \widetilde{ {\rm U}(n,1)} \; . $$  
This gives rise to the above (non-faithful) homogeneous presentation
of  the universal covering space $X$ for $Y$ in the diagram. It also follows  that
$$  \Psh(X)^{0} =   \widetilde{ {\rm U}(n,1)} \big/ \cZ  =  \,  { \widetilde{{\rm SU}(n,1)} \cdot  \widetilde{{\rm U}(1)} } \; . $$ 
\end{remark} 

A pseudo-Hermitian structure $\{\om,J\}$ 
on $Y = X \big/ \ZZ$ is obtained as a  
\emph{connection bundle} over $\HH^n_\CC$  such that   
$ P^*\Omega=d\om $, 
for the K\"ahler form $\Omega$ of  $\HH^n_\CC$. Here $S^1$ becomes the Reeb flow for $\om$ on $Y$, and 
{$$\Psh(Y)^{0} =  {\rm U}(n,1) \; . $$}
The pseudo-Hermitian structure $(\tilde \om, J)$ on $X$ is a lift of $\om$. 
Note also that $$ Y =  {\rm SU}(n,1)\big/ {\rm SU}(n) \text{ and }  X =  \widetilde{ {\rm SU}(n,1)}\big/ {\rm SU}(n)\;  $$
are \emph{faithful} presentations as homogeneous Sasaki manifolds of simple Lie groups. 
Taking a torsionfree discrete uniform subgroup
$\Gamma$ of $ {\rm SU}(n,1)$ (such a subgroup exists by \cite{borel}, for example),  
gives rise to a  \emph{regular}  locally homogeneous aspherical Sasaki manifold with Boothby-Wang fibering 
\begin{equation}\begin{CD}
S^1@>>>\Gamma \, \big\backslash \, {\rm SU}(n,1)\big/\, {\rm SU}(n)@>>>Q \, \backslash \HH^n_\CC \; , 
\end{CD}\end{equation}
 where $Q\leq {\rm PU}(n,1)$ is a  torsionfree  discrete uniform subgroup (isomorphic to $\Gamma$). 

\subsection{Join of locally homogeneous Sasaki manifolds} 
As above   let $$X_{\CC} =   \widetilde{{\rm SU}(n,1)} \big/ {{\rm SU}(n)} $$ 
denote the contractible Sasaki homogeneous space over $\HH^{n}_{\CC}$.  
(Compare Section \ref{sect:hnexample}).  We may take the join (see Proposition \ref{pro:join}) with the Sasaki Heisenberg group $\cN$ to obtain a contractible homogeneous Sasaki manifold: 
\begin{equation*}\label{eq:hypbun}
\begin{CD}
\RR@>>>  X = \; (\cN \times X_{\HH^n_\CC}) \big/ \Delta @>{\sf q}>>\CC^k\times \HH^n_\CC\\
@. ||@. ||@. \\
\RR@>>> \left(\cN \cdot  \widetilde{{\rm SU}(n,1)}\right) \big/ \, {\rm SU}(n)
@>{\sf q}>>\CC^k\times  {\rm SU}(n,1) \big/\; { {\rm S}  \left( {\rm U}(n)\times{\rm U}(1) \right) } 
\end{CD}\end{equation*}
A pseudo-Hermitian structure $\{\om,J\}$
on $$  \cN *  X_{\HH^n_\CC} =  (\cN \times X_{\HH^n_\CC}) \big/ \Delta$$
is obtained as the quotient of $\om_0+\tilde \om$, where $\omega_{0}$ is the 
 contact form on $\cN$, $\tilde \om$ on  $X_{\HH^n_\CC}$ (see Proposition \ref{pro:join}). 
Taking a suitable torsionfree discrete uniform subgroup $\pi$ from
\[\Psh(X)^0=(\cN\rtimes {\rm U}(k)) * \Psh(X_{\HH^n_\CC})  =
(\cN\rtimes {\rm U}(k))\cdot  \widetilde{{\rm SU}(n,1)}\]
allows to construct a compact locally homogeneous aspherical Sasaki 
manifold over a product of compact K\"ahler manifolds:
\begin{equation*}
\begin{CD}
S^1@>>> \pi\, \backslash (\cN *X_{\HH^n_\CC})@>{\sf q}>>
T_\CC^k\times Q\backslash \HH^n_\CC \:  .
\end{CD}\end{equation*}

\subsection{Heisenberg Sasaki manifolds} 
%
%
Recall from the construction in \eqref{eq:cnkl} that the Sasaki Lie groups  $$ \cN(k,l)$$  are contained in the pseudo-Hermitian group $\Psh(\cN) = \cN \rtimes \mathrm{U}(k)$ of the Heisenberg Sasaki group $\cN$.  Therefore, taking  quotients of 
 $\cN(k,l)$ by discrete uniform subgroups gives rise to: 

\smallskip 
\paragraph{\em Circle bundles over flat K\"ahler manifolds} 
Let $\Delta$ be a discrete uniform subgroup of $\cN(k,l)$. Then 
$$  M = \; \Delta \backslash \, \cN(k,l) $$ 
is a locally homogeneous $\cN(k,l)$-manifold. Since $\cN(k,l) \leq \Psh(\cN)$ acts simply transitively on $\cN$, $\Delta \leq  \Psh(\cN)$ acts properly discontinuously  as a discrete group of holomorphic 
isometries 
on $\cN$. Therefore $$ M  =  \, \cN \big/ \, \Delta $$
is also quotient of $\cN$ as a locally homogeneous manifold modeled on the homogeneous space  $\cN$. 
Moreover, the proof of Theorem \ref{undrphi} part (3), 
together with the exact sequence  \eqref{eq:cnkl},  show that 
$\Delta$ is a central extension of $p(\Delta)$, where $p(\Delta)$ is a uniform lattice
in  $\CC(k,l)$.  
This gives rise to a circle bundle 
$$   S^{1} \to  \; \Delta \backslash \cN  \; \to \; \, p(\Delta)\backslash \CC(k,l)  \; , $$ 
where the K\"ahler solvmanifold $p(\Delta)\backslash \CC(k,l)$ is a torus bundle over a torus,  and it is finitely covered  by a complex compact torus $T^n_\CC = \CC^{n}/\Lambda$, $\Lambda$  
isomorphic to $\ZZ^{2n}$ (compare \cite{ha}), where the  K\"ahler metric on  
$T^n_\CC$ is flat. 

\subsection{Locally homogeneous manifold  $\pi \backslash \cN$ which is not Sasaki} \label{sect:nonS}

We explicitly construct an example of  a Riemannian metric which is locally a Sasaki metric but does not admit a compatible structure vector field $\cA$.
(See also $(7)$ in the introduction, following Remark \ref{rem:seifert}).

\begin{exam} \label{ex:quotient_tau} 

Let $$ \Lambda = \ZZ \times (\ZZ^{n} + i \ZZ^{n} )\,  \subseteq \, \cN = \RR \times  \CC^{n}$$  be the integral 
lattice in $\cN$. Clearly, $\Lambda$ is a subgroup and $\tau \Lambda = \Lambda$, where as in \eqref{eq:tau}, 
$$  \displaystyle \tau(t,z)=(-t,\bar z) . $$ 
Next put $\alpha_s =(0, (s,0,\ldots,0))$,  $\mu = \alpha_{1 \over 2} \tau$ and let $$ \pi = \langle \mu, \Lambda \rangle \leq \; \cN \rtimes \tau$$
be the group generated by $\mu$ and $\Lambda$. Since $\mu^{2} = \alpha_{1} \in \Lambda$ and $\mu \Lambda \mu^{-1}  =  \Lambda$,  the group $\pi$ satisfies 
an exact sequence $$ \displaystyle 1\ra\, \Lambda \ra\, \pi \to \, \ZZ_2\to 1 \; . $$
Since $\mu$ is of infinite order $\pi$ must be torsionfree (see Lemma \ref{lem:tf} below).
\end{exam}

\begin{lemma} \label{lem:tf} 
$\pi$ is torsion-free. 
\end{lemma}
\begin{proof}
Recall that every non-trivial element of $\cN$ has infinite order.
Let $\gamma = \gamma_{0} \tau$, where $\gamma_{0} \in \cN$. If $\gamma$ has finite order, 
so has $\gamma^2 =  \gamma_{0} \gamma_{0}^{\tau} \in \cN$. Thus $\gamma^2 = 1 \in \cN$. Writing 
$\gamma_{0} =(t, w)$, we have by \eqref{eq:glN} and  \eqref{eq:tau} that 
$$ \gamma^2  = (t, w) \cdot (-t, \bar w) = ( -{\rm Im}({}^t\bar w \bar w), w + \bar w) = (0,0) \; . $$ 
That is, $\gamma$ is a torsion element, if and only if $w$ is purely imaginary. 
Assuming now that $\gamma \in \pi$, we have $\gamma_{0} = \lambda \alpha_{1 \over 2}$, where 
$\lambda \in \Lambda$ is integral. This shows that the vector 
$w$ for $\gamma_{0}$ has a non-trivial real part (in its first entry). Hence, $\gamma$ is not a torsion element. 
So $\pi$ is torsionfree,
\end{proof}

Since $\pi$ is without torsion, the quotient space
$$    \pi \, \backslash \, \cN $$
is a compact infra-nilmanifold. 
Since $\pi \leq  \Isom(\cN,g_{0})$, for the Sasaki metric $g_0$ on $\cN$ (as in Example \ref{ex:SasakiN}), 
there is an induced Riemannian metric $\hat g_0$ on $\pi \backslash \cN$, which is locally the same as the 
Sasaki metric $g_{0}$. But $(\pi \backslash \cN,\hat g_0)$ never admits a compatible Sasaki structure. That is,
there exists no pseudo-Hermitian structure $(\hat \eta, \hat J')$ on $ \pi \backslash \cN$ such that 
$\displaystyle 
\hat g_0=\hat\eta\cdot \hat\eta+d\hat\eta\circ \hat J'$: 

\begin{lemma}\label{nosasaki}
The infra-nilmanifold $(\pi\backslash \cN,\hat g_0)$ does not admit any
compatible Sasaki structure.
\end{lemma}
\begin{proof}
Suppose $(\pi\backslash \cN,\hat g_0)$
admits a Sasaki structure $(\hat\eta, \hat J')$ such that $\displaystyle 
\hat g_0=\hat\eta\cdot \hat\eta+d\hat\eta\circ \hat J'$.
Let $\eta$ be a lift of $\hat\eta$ to $\cN$,  for which
$\displaystyle 
g_0=\eta\cdot \eta+d\eta\circ J'$ is a Sasaki metric on $\cN$. 
Moreover, 
\begin{enumerate}
\item[{\rm (1)}]  $(\eta,J')$ is a standard pseudo-Hermitian structure on $\cN$.
\item[{\rm (2)}] $\pi\leq \Psh(\cN, \{\eta,J'\})\leq \Isom(\cN,g_0)=
\cN\rtimes ({\rm U}(k)\rtimes \ZZ_2)$.
\item[{\rm (3)}] with respect to the inclusion in (2), $\pi$ maps onto  $\ZZ_2$.
\end{enumerate}
Let $T'$ be the one-parameter group of the Reeb field for $\eta$.
As $T'$ 
is contained in the isometry group of $g_{0}$, and $T'$ 
is connected, it follows  that $$ T'\leq \cN\rtimes {\rm U}(k)$$  by (2). 
In particular,  $T'$ normalizes $\cN$. 
Since $T$ is the lift of the Reeb flow on $\pi\backslash  \cN$, 
it centralizes $\pi$ and $\pi\cap \cN$ (also by (2)). 
Since $\pi\cap \cN$ is discrete uniform in $\cN$ (by 
the Auslander-Bieberbach theorem \cite{ab}),
$T'$ centralizes $\cN$ by the Mal'cev unique extension property.
Since $T'\leq \cN\rtimes {\rm U}(k)$,
this implies 
$T'=C(\cN)=T$ is  the one-parameter subgroup of the Reeb field for $\om_0$.
As $\pi$ centralizes $T$, it follows
$\pi\leq \cN\rtimes{\rm U}(k)$.
This contradicts (3).
\end{proof} 

Therefore the  compact locally homogeneous aspherical manifold
$\pi\backslash \cN$ admits a locally Sasaki metric but it is not a Sasaki manifold. In addition
$\Isom(\pi\backslash \cN, \hat g_{0})$ is finite, and 
$\pi\backslash \cN$ is an $S^1$-fibred infranil-manifold without
any $S^1$-action. 

\end{document}